\documentclass[11pt,reqno,a4paper]{amsart}
\usepackage[normalem]{ulem}
\usepackage[colorlinks]{hyperref}
\usepackage{amssymb,epsfig,graphics}
\usepackage{amsmath}
\usepackage{dsfont}
\usepackage{enumerate}
\usepackage{tikz}
\usetikzlibrary{arrows}

\newcommand{\vertiii}[1]{{\left\vert\kern-0.25ex\left\vert\kern-0.25ex\left\vert #1 
    \right\vert\kern-0.25ex\right\vert\kern-0.25ex\right\vert}}

\newtheorem{theorem}{Theorem}[section]
\newtheorem{proposition}[theorem]{Proposition}
\newtheorem{definition}[theorem]{Definition}
\newtheorem{corollary}[theorem]{Corollary}
\newtheorem{lemma}[theorem]{Lemma}

\newtheorem{remark}[theorem]{Remark}
\newtheorem{example}[theorem]{Example}

\theoremstyle{definition}

           \newcommand\vfp{\vartheta}

\def\L{\mathcal{L}}
\def\C{\mathcal{C}}

\def\M{\mathcal{M}}

\def\KK{\mathbb{K}}
\def\NN{\mathbb{N}}

\def\RR{\mathbb{R}}
\def\TT{\mathbb{T}}

\def\CC{\mathbb{C}}
\def\FF{\mathbb{F}}

\DeclareMathOperator{\divv}{div}
\DeclareMathOperator{\tr}{Tr}
\newcommand\vf{\varphi}
\def\M{\mathcal{M}}

\let\eps=\varepsilon

\def\B{\mathcal{B}}

\def\RR{{\mathbb R}}

\def\1{{{\mathit 1} \!\!\>\!\! I} }

\newcommand{\param}{{ \varpi}}

\DeclareMathOperator{\Diff}{Diff}

\numberwithin{equation}{section}

\newcommand{\cF}{\mathcal F}
\newcommand{\cC}{\mathcal C}
\newcommand{\cM}{\mathcal M}
\newcommand{\bN}{\mathbb N}
\newcommand{\bR}{\mathbb R}
\newcommand{\ve}{\varepsilon}
\newcommand\Id{{\mathds{1}}}

\newcommand\boeta{\boldsymbol{\eta}}
\newcommand\btheta{\boldsymbol{\theta}}

\begin{document}
\title{globally coupled Anosov diffeomorphisms:\\ 
Statistical properties}
\author{Wael Bahsoun}
\address{Department of Mathematical Sciences, Loughborough University,
Loughborough, Leicestershire, LE11 3TU, UK}
\email{W.Bahsoun@lboro.ac.uk}
\author{Carlangelo Liverani}
\address{
Dipartimento di Matematica\\
II Universit\`{a} di Roma (Tor Vergata)\\
Via della Ricerca Scientifica, 00133 Roma, Italy.}
\email{{\tt liverani@mat.uniroma2.it}}
\author{Fanni M. S\'elley}
\address{Mathematical Institute of Leiden University, Niels Bohrweg 1
2333 CA Leiden, The Netherlands}
\email{f.m.selley@math.leidenuniv.nl}
\thanks{The research of W. Bahsoun is supported by EPSRC grant EP/V053493/1.  C. Liverani acknowledges the MIUR Excellence Department Project awarded to the Department of Mathematics, University of Rome Tor Vergata, CUP E83C18000100006, and the INDAM-GNFM. The paper was partially supported by the Grant PRIN 2017S35EHN. W. Bahsoun would like to thank the hospitality of Leiden University, while C. Liverani and F. S\'elley would like to thank the hospitality of Loughborough University during the course of this work.\\
}
\begin{abstract}
We study infinite systems of globally coupled Anosov diffeomorphisms with weak coupling strength. Using transfer operators acting on anisotropic Banach spaces, we prove that the coupled system admits a \emph{unique physical} invariant state, $h_\eps$. Moreover, we prove \emph{exponential convergence to equilibrium} for a suitable class of distributions and show that the map $\eps\mapsto h_\eps$ is Lipschitz continuous. 
\end{abstract}
\date{\today}
\maketitle
\markboth{W. Bahsoun C. Liverani and F. M. S\'elley}{Globally Coupled Anosov Diffeomorphisms}
\bibliographystyle{plain}

\section{Introduction}
Coupled systems are mathematical models of spatially extended systems consisting of identical interacting units. They provide a challenging subject of study from a mathematical point of view, give a well-motivated example of infinite dimensional dynamical systems, and often exhibit phase transition-like parameter-dependent behavior. Their popularity stems from the fact that they describe considerably well real-world systems (e.g. coupled oscillator networks \cite{Kur84, BA11}, heterogeneous networks \cite{P20} and networks with higher order interactions \cite{Bi21}). 

In the field of dynamical systems coupled maps were introduced by Kaneko \cite{Kan93} and were first studied rigorously by Bunimovich and Sinai \cite{BS88} in the case of \emph{nearest neighbhor} interacting smooth expanding maps. The results of \cite{BS88} were later extended to piecewise expanding maps \cite{KL04, KL09} and to coupled map lattices where the dynamics on each site is given by a smooth Anosov map \cite{PS91}.
 
An important type of coupled systems, which is different from the coupled map lattice model, is the {\em globally coupled} or {\em mean field model}. For example, gas particles interacting via their mean field give rise to the so-called Vlasov equation, \cite{V68, Ga} or see \cite[Chapter 5]{Spo} for a simple derivation. In the case of plasma, a similar procedure gives the Vlasov-Poisson equation. Other examples of mean field models are the vorticity formulation of the two-dimensional Euler equation for incompressible fluids \cite{Euler}, and the time-dependent Hartree equation in quantum mechanics \cite{Hartree} (see \cite{G16} for more details). The study of the corresponding limiting equation is far from obvious as it may exhibit unexpected phenomena, e.g. Landau damping \cite{MV}. 

The above examples consider individual dynamics given by simple integrable motion. In the case in which the individual systems are strongly chaotic, statistical properties of the long-term behavior may be available. This was first shown by Keller \cite{K00} in the case of a toy model of globally coupled expanding maps. In this work, we study, for the first time, statistical properties of \emph{infinite systems of globally coupled} Anosov diffeomorphisms that are motivated by considering an appropriate limit of finitely many coupled Anosov maps.

As already mentioned, the topic of infinitely many globally coupled maps was pioneered by Keller \cite{K00}. In \cite{K00} the local dynamics was described by an expanding circle map or a piecewise expanding map of the interval. For such coupled systems Keller \cite{K00} proved, in the case of weak coupling strength, the existence of a unique invariant state and exponential convergence to equilibrium (see also \cite{Blank} for a similar result). In \cite{BKZ} a globally coupled system with site dynamics given by expanding fractional linear interval maps was studied. It was shown that the system undergoes a supercritical pitchfork bifurcation from a unique stable equilibrium to the coexistence of two stable and one unstable equilibrium. Both \cite{K00} and \cite{BKZ} consider a coupling that only involves a parameter computed from the system state according to some fixed scheme. Later in \cite{BKST18} the work of \cite{K00} was extended to a more general coupling that mimics elastic interaction on the circle \cite{B14}. In addition to the analogous results of \cite{K00}, Lipschitz continuity of the equilibrium state, as a function of the coupling strength parameter, was proved. This was taken one step further in \cite{ST21} where linear response was shown in a rather general, smooth setting. Recent advances on globally coupled maps can be found in \cite{G21}. The work of \cite{G21} includes an abstract framework and applications to study statistical aspects of globally coupled circle maps. However, up to date there are no ergodic theoretic results on globally coupled higher dimensional hyperbolic systems in an infinite limit. This is because the right functional analytic tools to study hyperbolic systems, namely transfer operators acting on appropriate anisotropic Banach spaces, were not available until recently. 

Starting with the paper \cite{BKL}, there has been a growing interest in developing anisotropic Banach spaces and spectral properties of transfer operators associated with hyperbolic dynamical systems. The books \cite{Ba, DKL21} provide an extensive account of the topic. 

Recently, a new family of anisotropic Banach spaces was introduced in \cite{BL21} which are not only amenable to perturbations of the dynamics\footnote{See the earlier work of \cite{GL, GL1, BT, DZ15} where Banach spaces that are amenable to perturbations of the dynamics were also constructed. See also \cite{GaL} for Banach spaces amenable to perturbations, although limited to skew products. Similarly to the present case, the Banach spaces of \cite{GaL} also resemble $L^1$ and $BV$.}, but for which the weak norms `behave' like $L^1$ and the strong norms `behave' like BV, the space of functions of bounded variations (\cite[Remark 2.15]{BL21} for more details). These properties significantly simplify the study of the long-term behavior of the iterates of transfer operators associated with the globally coupled Anosov systems both for finite systems and in the mean field limit. Indeed, these spaces allow to define a \emph{simple} invariant set of distributions under the action of the associated transfer operator (see  \eqref{eq:bk}), where an invariant state of the system is proven to exist. 

Yet, to prove finer statistical properties of globally coupled Anosov maps, we must introduce a higher order version of the spaces in \cite{BL21} (see \eqref{eq:highernorms}). An important feature of this new space is that the transfer operator associated with the globally coupled system admits exponential memory loss with respect to its \emph{weak norm}. Using this information we prove the uniqueness of the invariant state, $h_\eps$, exponential convergence to equilibrium and that the map $\eps\mapsto h_\eps$ is Lipschitz continuous (see Theorem \ref{thm:unifix} below). This information allows proving that such measure is the \emph{unique physical measure} of the system (see Theorem \ref{thm:main0}). 

The paper is organised as follows. In section \ref{sec:mapstat} we introduce our system, state our main results (Theorem \ref{thm:main0} and Theorem \ref{thm:unifix}) and provide a strategy of the proof. In section \ref{sec:proofs} we prove Theorem \ref{thm:unifix} in a series of lemmas and propositions. Theorem \ref{thm:main0} is proved in the same section. Appendix A includes results on adapted foliations and test functions needed for the Banach spaces used in the paper. Appendix B includes statements and proofs about perturbations of Anosov maps. Appendix C includes properties of projections along the unstable direction, which are needed in the proof of the Lasota-Yorke inequality in Lemma \ref{lem:SLY}.

\section{The system and the statement of the main result}\label{sec:mapstat}
\subsection{The individual map.}\label{subsec:map}\ \\ 
 Let $d \geq 2$ and consider a $d$-dimensional compact manifold $M$. Define the differentiable structure by the open cover $\{V_i\}_{i=1}^S$ and charts $\phi_i: V_i \to \RR^d$, $\phi_i \in \C^r$ for some $r \geq 4$. More precisely, consider a fixed smooth partition of unity $\{\vfp_i\}$ subordinated to $\{V_i\}_{i=1}^S$ and define a smooth volume form $\omega$ by
\begin{equation*}
\int_Mh\; d\omega=\sum_{i=1}^S\int_{\phi_i(V_i)}h\circ  \phi_i^{-1}(z)\;\vfp_i\circ \phi_i^{-1}(z) d z.
\end{equation*}
All integrals will be understood with respect to such a form from now. 

Consider an Anosov diffeomorphism $T \in \Diff^r(M)$, $r>1$; i.e., there exists $\lambda_0>1, \nu_0\in (0,1)$, $c_0\in(0, 1)$ and a continuous cone field $\C=\{C(\xi)\}_{\xi\in M}$, $ \overline{C(\xi)}=C(\xi)\subset T_\xi M$ such that $D_{\xi}T^{-1} C(\xi)\subset \text{int}(C(T^{-1}(\xi)))\cup\{0\}$ and
\begin{equation}\label{eq:anosov}
\begin{split}
&\inf_{\xi\in M}\inf_{v\in C(\xi)} \|D_\xi T^{-n}v\|> c_0\nu_0^{-n}\|v\|\\
& \inf_{\xi\in M}\inf_{v\not\in C(\xi)} \|D_\xi T^{n}v\|> c_0\lambda_0^{n}\|v\|.
\end{split}
\end{equation}
We will sometimes refer to it as the stable cone field (and the unstable cone field will be the complement.)
We also assume that $T$ is transitive.
\subsection{Motivation for infinite coupled map systems}\ \\ 
Denote the $N$-fold products $T \times \dots \times T$ and $M\times\cdots\times M$ by $T_N$ and $M_N$, respectively. We view $(T_N,M_N)$ as a system of $N$ units, (called either sites or particles in the coupled maps literature) each with a state in $M$ evolving in time according to $T$. Define a diffeomorphism $\Phi^{\eps}_N$ of the product manifold $M_N$, $\eps$ close to $Id_{M^N}$. We interpret $(T_N\circ \Phi^{\eps}_N, M_N)$ as a coupled system of $N$ interacting units, where $\Phi^{\eps}_N$ accounts for the interaction between individuals, with strength tuned by the parameter $\eps$-- in particular we assume that $\Phi^{0}_N=Id_{M^N}$.

The system state can be described by the vector $(x_1,\dots,x_N)$, or equivalently, by the empirical measure $\frac{1}{N}\sum_{i=1}^N \delta_{x_i}$. More precisely calling $\cM_1(M)$ the set of probability measures over $M$, we can define the natural embedding $\Psi_N: M_N\to \cM_1(M)$ given by $\Psi_N(x)=\frac{1}{N}\sum_{i=1}^N \delta_{x_i}$. Let 
$$F_\ve:M\times \cM_1(M)\to M$$
be $\cC^r$, $r>3$, in the first variable and continuous (with respect to the weak topology) in the second variable uniformly in $x$. We assume that the coupling has the form ({\em mean field coupling})
\[
(\Phi^{\eps}_N(x))_i = F_{\eps}\left(x_i,\Psi_N(x)\right).
\]
For $\mu\in\cM_1(M)$, define 
$$\Phi^{\eps}_\mu =  F_{\eps}(\cdot,\mu).$$ 
Note that, if $\mu_N=\frac{1}{N}\sum_{i=1}^N \delta_{x_i}$, then
\begin{equation}\label{eq:finite_dyn}
(T\circ \Phi^{\eps}_{\mu_N})_*\left(\frac{1}{N}\sum_{i=1}^N \delta_{x_i}\right)=\frac{1}{N}\sum_{i=1}^N \delta_{T\circ \Phi^{\eps}_{\mu_N}(x_i)}
\end{equation}
and, hence, in this case the dynamics on $M_N$ induces a dynamics on $\cM_1(M)$. Such dynamics extends naturally on all $\cM_1(M)$. Thus, in the case in which $\mu$ is a probability measure with a density, the map 
$$(T\circ \Phi^{\eps}_{\mu})_*:\cM_1(M)\to\cM_1(M)$$ 
can be interpreted as the evolution of a state with infinitely many interacting units with state distribution given by $\mu$. Indeed, given a sequence of empirical measures $\mu_N=\frac{1}{N}\sum_{i=1}^N \delta_{x_i}$ converging in the weak topology to some measure $\mu$, as $N \to \infty$, we have by hypothesis
\begin{equation} \label{eq:thmdynlim}
F_\eps\left(\cdot,\frac{1}{N}\sum_{j=1}^N \delta_{x_j}\right) \to F_{\eps} (\cdot,\mu), \qquad N \to \infty,
\end{equation}
where the convergence is in the uniform topology.
Then, recalling \eqref{eq:finite_dyn}, for each $\vf\in\cC^0(M)$ we have
\[
\begin{split}
\lim_{N\to\infty}(T\circ \Phi_{\eps,\mu_N})_*\mu_N(\vf)
&=\lim_{N\to\infty}\frac{1}{N}\sum_{i=1}^N \vf(T\circ \Phi_{\eps,\mu_N}(x_i)).
\end{split}
\]
By \eqref{eq:thmdynlim}, $T\circ \Phi_{\eps,\mu_N}\to T\circ \Phi_{\eps,\mu}$ uniformly, so for each $\delta>0$ there exists an $N_\delta>0$ large enough such that, for all $N\geq N_\delta$ it holds true 
\[
\sup_i |\vf(T\circ \Phi_{\eps,\mu_N}(x_i))-\vf(T\circ \Phi_{\eps,\mu}(x_i))|\leq \delta.
\]
Thus we can write
\[
\begin{split}
\lim_{N\to\infty}(T\circ \Phi^{\eps}_{\mu_N})_*\mu_N(\vf)
&=\lim_{N\to\infty}\frac{1}{N}\sum_{i=1}^N \vf(T\circ \Phi^{\eps}_{\mu}(x_i))=\mu(\vf\circ  T\circ \Phi^{\eps}_{\mu})\\
&=(T\circ \Phi^{\eps}_{\mu})_*\mu(\vf).
\end{split}
\]
This construction provides a motivation for the systems studied in the paper including the examples of subsection \ref{sub:ex}. 

\subsection{Statement of the main results}\ \\
Our aim is to study the long-time behavior of the self-consistent evolution $\mu \mapsto (T\circ \Phi^{\eps}_{\mu})_*\mu$. 
In particular, we are interested in classifying the invariant measures and their stability properties.

Yet, the above setting and question are too general to allow a precise answer. We will thus introduce two further technical assumptions\footnote{These assumption are essentially saying that $\Phi^\varepsilon_h$ is close to the identity, both in $\varepsilon$ and $h$, in an appropriate topology.} \eqref{eq:CouplingAssum1}, \eqref{eq:CouplingAssum2} on the coupling that will be detailed in subsection \ref{sec:coup}.
As for the imprecision of the task, it arises from the possibility of having physically irrelevant invariant measures. For example, measures that represent a finite number of particles or that describe statistical properties we are not interested in. This problem appears already in the study of the invariant measures of an Anosov map and a typical solution is to restrict to physical measures. We, therefore, introduce an analogous definition for physical measures in the present situation of infinitely globally coupled systems.

 \begin{definition}\label{def:phys}
We call a measure $h_\eps\in \cM_1(M)$ invariant if 
\[
(T\circ \Phi^{\eps}_{h_\eps})_* h_\ve=h_\eps.
\]
Moreover, we call an invariant measure $h_\eps$ physical if there exists some $h\in L^1$, such that $\mu_0=hd\omega\in \cM_1(M)$, and, defining for each $n\in\bN\cup\{0\}$, 
\[
\mu_{n+1}=(T\circ \Phi^{\eps}_{\mu_n})_* \mu_n
\]
the sequence $\{\mu_n\}$ converges weakly to $h_\eps$.
\end{definition}
\begin{remark}
In essence, physical measures are measures that the system can asymptotically attain when starting with an initial condition that is absolutely continuous with respect to Lebesgue.
\end{remark}
Our first main result is as follows.
\begin{theorem}\label{thm:main0}
 Under assumptions  \eqref{eq:CouplingAssum1},  \eqref{eq:CouplingAssum2} there exists $\eps_0>0$ such that, for all $\eps<\eps_0$ the system admits a unique physical measure $h_\eps$. 
\end{theorem}
Theorem \ref{thm:main0} is the consequence of a more quantitative result, Theorem \ref{thm:unifix}.  To prove Theorem \ref{thm:unifix} we need to introduce a more suitable topology. This is done, in analogy with the strategy used for Anosov maps and flows, by introducing Banach spaces adapted to the dynamics.

\subsection{Anisotropic BV}\ \\ 
The following anisotropic Banach spaces, introduced in \cite{BL21}, will play a crucial role in this paper. To define these spaces, we need to consider appropriate foliations of $M$ and test functions suited to such foliations. The spaces $\Omega_{L,q,l}$ collect pairs $(W,\varphi)$ where $W$ is a foliation, $\varphi$ is a test function on $M$ with controlled regularity on $W$, while the labels of the $L,q, l$ are numbers: $L>0$ is a uniform bound on some regularity class of the foliations, $q\in\mathbb N$ is the number of derivatives we consider along the stable direction and $l$ is the dimension of the target Euclidean space of $\varphi$. For a precise definition see \eqref{eq:measure} in the Appendix \ref{sec:foliation} and consult \cite{BL21} for a detailed discussion. Given a function $h\in\C^1(M,\CC)$ we define
\begin{equation}\label{eq:norms}
\begin{split}
&\|h\|_{0,q}:=\sup_{(W,\varphi)\in\Omega_{L,q,1}}\left |\int_M h \, \varphi \right |\\
& \|h\|^*_{1,q}:=\sup_{(W,\varphi)\in\Omega_{L,q+1,d}}\left |\int_M  h \, \text{div}\varphi\right |\\
&\|h\|^-_{1,q}:=a\|h\|_{0,q}+\|h\|^*_{1,q},
\end{split}
\end{equation}
for any $q\in\NN\cup\{0\}$ and some fixed $a>0$. Let $\B^{0,q}$ be the Banach space obtained by completing $\C^1(M,\RR)$ in the $\|\cdot\|_{0,q}$ norm. For each $h\in\B^{0,q}$ let
\begin{equation}\label{eq:norm_bv}
\|h\|_{1,q}=\lim_{\epsilon\to 0}\inf\{\|g\|^-_{1,q} : g\in \C^1(M,\RR)\text{ and }\|g-h\|_{0,q}\leq\epsilon\}.
\end{equation}
We then define $\B^{1,q}:=\{h\in\B^{0,q}\;|\; \|h\|_{1,q}<\infty\}$.

\begin{remark}
According to \cite[Lemma 2.12]{BL21}, there exists a canonical continuous injective map $\iota:\B^{0,q}\to (\C^q)'$.
In the following, we will use $\iota$ to identify a positive element $h\in \B^{0,q}$ with the measure $h d\omega=\iota(h)$ without any further comment. The next lemma further clarifies this.
\end{remark}
\begin{lemma}
A positive element of $\B^{0,q}$ is a measure. In addition, when restricted to $\cM_1(M)\cap\B^{0,q}$, where $\cM_1(M)$ denotes the set of probability measures over $M$, the norm $\|\cdot\|_{0,q}$ is identical to $\|\cdot\|_{TV}$, the total variation norm.
\end{lemma}
\begin{proof}
The first claim is standard as positive distributions are measures.\footnote{ If $h\in (\C^{s})'$, $s>0$, is positive, then for each $\vf\in \C^\infty$ we have $h(\|\vf\|_\infty\pm\vf)\geq 0$, so $|h (\vf)|\leq h(1)\|\vf\|_\infty$ and the claim follows by the Riesz theorem.}
To conclude, note that, since  $\|\mu\|_{TV}=\sup_{\vf\in\C^0}\int\vf d\mu$, if $d\mu= h d\omega$, we have $\|h\|_{0,q}\leq \|\mu\|_{TV}$. On the other hand
\[
\|\mu\|_{TV}=\int_M h dx\leq \|h\|_{0,q}.
\]
\end{proof}
For further use, we need to define a stronger norm, extending the spaces in \cite{BL21}. For each $h\in\C^2(M,\CC)$, define
\begin{equation}\label{eq:highernorms}
\begin{split}
\|h\|^*_{2,q}&:= \sup_{\underset{ i=1,\dots, d}{(W,\varphi^i)\in\Omega_{L,q+2,d}}} \sum_{i=1}^{d}\int_M  \partial_{x_i}h \, \text{div}\varphi^i
;\\
\|h\|_{2,q}&:= b\|h\|_{1,q}+\|h\|^*_{2,q}.
\end{split}
\end{equation}
Let $\B^{2,q}$ be the Banach space obtained by completing $\C^2(M,\RR)$ in the $\|\cdot\|_{2,q}$ norm. 

\subsection{Assumptions on the coupling.} \label{sec:coup}\ \\ 
We are now able to specify precisely our assumption on the coupling.
Let us define $\B_1^{0,q}=\cM_1(M)\cap\B^{0,q}$.

We assume the coupling satisfies the following two conditions:
\begin{align}
d_{C^{r}}(\Phi_{h_1}^{\eps},\Phi_{h_2}^{\eps})&\leq C |\varepsilon| \|h_1-h_2\|_{0,r-1}, & \text{for all }h_1,h_2 \in \B_1^{0,r-1}, \eps \in \RR \tag{A1}; \label{eq:CouplingAssum1} \\
d_{C^{r}}(\Phi_{h}^{\eps},\Phi_{h}^{\eps'})&\leq C |\eps-\eps'|, & \text{for all } h\in \cM_1(M), \eps, \eps' \in \RR. \tag{A2} 
\label{eq:CouplingAssum2}
\end{align}

As already explained, we define the coupled map as 
\begin{equation}\label{eq:coupledsys}
T_h^{\eps}=T \circ \Phi_{h}^{\varepsilon}
\end{equation}
for $h \in \B_1^{0,q}$ and $\varepsilon \in \RR$.  This map represents the dynamics of a globally coupled map in the so-called thermodynamic limit with site dynamics $T: M \to M$, system state given by the distribution $h$ and coupling strength $\eps$.
\begin{remark}\label{rem:anosov} Note that equation \eqref{eq:CouplingAssum2} implies that, for each $h \in \cM_1(M)$, we have
\[
d_{C^{r}}(T ,T_{h}^{\eps})\leq C|\eps|.
\]
Accordingly, there exists $\ve_0$ such that, for all $\ve\leq \ve_0$ the maps $T_h^\ve$ satisfy \eqref{eq:anosov} uniformly with the same cones (see Lemma \ref{lem:anosovcoup} for details).
\end{remark}

\subsection{Transfer operators}\ \\ 
We can now study the dynamics in the Banach spaces mentioned above. This is done by introducing a transfer operator acting on the an\-iso\-tropic BV spaces. 
Recalling that these spaces can be canonically embedded into $(\C^{q+i})'$ (according to \cite[Lemma 2.12]{BL21}, which can easily be extended to the case $i=2$), the transfer operator associated with $T$
$$\L_T:\mathcal B^{i,q}\to \mathcal B^{i,q}$$ can be defined as
\begin{equation}\label{eq:confused}
(\L_T h)\varphi= h(\varphi\circ T) \qquad \varphi \in \C^{q+i}, \quad i\in\{0,1,2\}.
\end{equation}

\begin{remark}\label{rem:anosv_stable}
Note that the precise version of \eqref{eq:confused} would be
\[
\iota(\L_T h)\varphi= (\iota h)(\varphi\circ T).
\]
As already remarked we allow the above imprecise notation \eqref{eq:confused} to simplify the notation.
In particular, when $h\in \C^1$, we identify $h$ with the measure $d\mu_h= h d\omega$ and the transfer operator associated to $T$ is then given by
\[
\L_T h = \frac{h}{|\det (DT)|} \circ T^{-1}.
\]
Clearly,
\[
d(T_*\mu_h)=(\L_T h) dx.
\]
\end{remark}
It follows that the evolution of the coupled system state is given by the \emph{self-consistent transfer operator} $\L_{\eps}: \B^{i,q}\cap \M_1(M) \to \M_1(M)$ that is defined as
\begin{equation} \label{eq:selfcL}
\L_{\eps}(h)=\L_{T_h^{\eps}}h
\end{equation}
where $\L_{T_h^{\eps}}$ is the transfer operator associated with the map $T_h^{\eps}$ defined in \eqref{eq:coupledsys}. 
Indeed, for $h\in L^1$
\[
\L_{T_h^{\eps}}h dx=(T^\ve_h)_*h.
\]
Notice that 
$$\L_{T_h^{\eps}}=\L_T\circ\L_{\Phi_h^{\eps}}.$$
Observe that unlike $\L_{T_h^{\eps}}$, the self-consistent transfer operator $\L_{\eps}$ is \emph{nonlinear}. 
Setting $h_n=\L_{\eps}^n(h_0)$, if well defined, we can write 
\[
\begin{split}
\L_{\eps}^n(h_0)&=\L_{T_{h_{n-1}}^{\eps}}\dots \L_{T_{h_1}^{\eps}}\L_{T_{h_0}^{\eps}}h_0\\
&= 
\L_{T_{h_{n-1}}^{\eps}\circ\cdots\circ T_{h_1}^{\eps}\circ T_{h_0}^{\eps}}h_0.
\end{split}
\]
Note that if $\L_\eps(h_\eps)=h_\eps$, then the coupled system admits an invariant state.
\subsection{A more quantitative result}\ \\
Our goal is to prove that the self-consistent transfer operator $\L_\eps$ admits a unique fixed point and exhibits exponential convergence to equilibrium for a certain class of distributions. To do so we first define a compact convex subset of the $\B^{0,q+1}$ space. For $K \geq 0$, define
\begin{equation} \label{eq:bk}
\B(K,q)=\left \{h \in \C^{1} \: : \:  h\ge 0,\; \int h=1,\; \| h\|_{1,q} \leq K \right\}
\end{equation}
and let $\overline{\B}(K,q)$ be the closure of $\B(K,q)$ with respect to the $\|\cdot\|_{0,q+1}$-norm.
The proof of the following proposition, which shows that $\L_\eps$ can be iterated and that $B(K,q)$ is eventually invariant, can be found in Section \ref{sec:Leps is well defined}.
\begin{proposition}\label{lem:invariance}
There exists $\ve^*_1>0$ such that, for all $\eps<\eps^*_1$ and $q>0$, 
\[
\L_\eps ( \B^{0,q}\cap \M_1(M)) \subset   \B^{0,q}\cap\M_1(M).
\]
There exists $N \in \NN$, and $K_{\min} \ge 0$ such that, for all all $|\eps| < \eps^*_1$, $n\geq N$ and  $K \geq K_{\min}$, $\L_{\eps}^n (\overline{\B}(K,q)) \subset \overline{\B}(K,q)$.
\end{proposition}
 The above is not enough to obtain uniqueness and exponential convergence. This is because, on the weak space $\B^{0,q}$, $\L_T$ does not admit a spectral gap. To overcome this hurdle we define a stronger Banach space than $\B^{1,q}$ and a set finer than $\overline{\B}(K,q)$ to obtain our desired result.  

Let $K_1 \geq K_{\min}$, $K_2 \geq 0$ and define
\begin{equation} \label{eq:bk2}
\B(K_1,K_2,q)=\left \{h \in \C^{2} \: : \:  h\in\B(K_1,q+1) ,\; \| h\|_{2,q} \leq K_2 \right\}.
\end{equation}
Let $\overline{\B}(K_1,K_2,q)$ be the closure of $\B(K_1,K_2,q)$ with respect to the $\|\cdot\|_{1,q+1}$-norm.
\begin{theorem}\label{thm:unifix}
There exists $\eps^*_2 > 0$, $K_*(\eps^*_2) \geq 1$,  such that for all $q \in \{1,\dots,r-3\}$, $|\eps| < \eps^*_2 $ and $K_2\ge K_*(\eps)$, the following holds:
\begin{itemize}
\item[(i)] $\L_\eps$ has a unique fixed point $h_\eps$ in $\overline{\B}(K_1,K_2,q)$. In addition, there exists $\gamma\in (0,1)$ and $C>0$ such that, for all $h\in \overline{\B}(K_1,K_2,q)$, we have
\[
\|\L_\eps^{n}(h)-h_\eps\|_{1,q+1}\leq C \gamma^n .
\]
\end{itemize}
Moreover, for all $|\eps|, |\eps'| < \eps^*_2 $,
\begin{itemize}
\item[(ii)] if $h_\eps$ and $h_{\eps'}$ are  the unique fixed points of $\L_{\eps}$ and $\L_{\eps'}$ respectively, then there exists $C > 0$ (depending on $\eps^*_2$ and $K_*$) such that $$\|h_\eps-h_{\eps'}\|_{1,q+1}\le C|\eps-\eps'|.$$
\end{itemize}
\end{theorem}
The proof of the above Theorem is postponed to section \ref{sec:unique}.
\begin{remark}
Item (ii) provides statistical stability of the coupled system. In particular, it implies that the map $\eps\mapsto h_\eps$, $|\eps|\in[0,\eps_2]$ is Lipschitz continuous. 
\end{remark}

The strategy to prove Theorem \ref{thm:unifix} is as follows: we first show that $\L_{\eps}^N|_{\overline{\B}(K,q)}$ is continuous in the weak norm $\|\cdot\|_{0,q+1}$ to conclude that $\L^N_\eps$ has at least one fixed point in $\overline{\B}(K,q)$. We then prove that $\L^N_\eps$  is a contraction when acting on $\overline{\B}(K_1,K_2,q)$, which gives the uniqueness and exponential convergence. We then prove that this fixed point is actually a unique fixed point of $\L_\eps$ itself. Finally, we prove Lipschitz continuity of $\eps \mapsto h_\eps$ by using the exponential convergence result and that $\eps \mapsto \L_\eps$ is Lipschitz in a proper sense. 

Note that the above strategy is natural when the transfer operator associated with the site dynamics admits a spectral gap on a Banach space. See \cite{G21} for a general strategy similar to the one we implement in this work.

\begin{remark} \label{rem:galatolo1}
If one wants to follow \cite{G21} in the Anosov setting, one has to choose the regularity in our spaces carefully. It seems to us that such a choice may then require more regularity on the map. Moreover, \cite{G21} assumes a `one step' Lasota-Yorke inequality (see assumption (Con1) in \cite{G21} required to obtain the exponential convergence to equilibrium) which seems a strong assumption in a hyperbolic setting (the constant $A$ that appears in our Lasota-Yorke inequality, does not only depend on the map, but also on the class of the foliations considered in our norms). Therefore, instead of verifying the assumptions of \cite{G21} which will force us to adding restrictive assumptions on $T$, we are going to pursue a different line of argument.  
\end{remark}
 
\subsection{Examples}\label{sub:ex}\ \\ 
Before proving Theorem \ref{thm:main0}, Proposition \ref{lem:invariance}, and Theorem \ref{thm:unifix}, we provide a class of examples that satisfy assumptions \eqref{eq:CouplingAssum1}-\eqref{eq:CouplingAssum2}. In both examples we consider $M=\TT^d$ and $\mu \in \cM_1(\TT^d)$.
\begin{example}

 Consider $\Phi_{\mu}^{\eps}$ given by the following formula:
\[
\Phi^{\eps}_{\mu}(x) = x + \eps \int_{\TT^d} \KK_1(x)\KK_2(y)d\mu(y)
\]
for some $\KK_1,\KK_2 \in C^{\infty}(\TT^{d},\TT^d)$.

Write this as
\[
\Phi_{\mu}^{\eps}(x)=x+\eps \gamma(x, \mu(\KK_2)),
\]
where
\[
\mu(\KK_2)=\int_{\TT^d}\KK_2(x)d\mu(x),
\]
$\gamma(x,\mu(\KK_2)))=\KK_1(x) \cdot \mu (\KK_2)$. Assume that $\int d\mu(x)=\int h(x)dx$ for some $h\in C^r(\TT^d)$. This is analogous to the one dimensional setting of \cite{K00}.

Since
\[
\partial^{\alpha}(\Phi^{\eps}_{h})_j=
\begin{cases}
1+\eps \cdot \partial^{\alpha}\KK_1 \cdot h(\KK_2) & \quad \text{if } \partial^{\alpha}=\partial_{x_j} \\
\eps \cdot \partial^{\alpha}\KK_1 \cdot h(\KK_2) & \quad \text{otherwise},
\end{cases}
\]
we have
\[
d_{C^{r}}(\Phi_{h}^{\eps},\Phi_{h}^{\eps'})\leq \|\KK_1\|_{C^r} |\varepsilon-\eps'| \left|\int_{\TT^d} \KK_2(x)h(x)dx\right| \leq  C(\KK_1,\KK_2)|\eps-\eps'| \|h\|_{0,r}
\]
and 
\[
d_{C^{r}}(\Phi_{h}^{\eps},\Phi_{h'}^{\eps})\leq  \|\KK_1\|_{C^r} |\varepsilon| \left|\int_{\TT^d} \KK_2(x)(h-h')(x)dx\right| \leq C(\KK_1,\KK_2)|\varepsilon| \|h-h'\|_{0,r}.
\]
\end{example}
\begin{example}
Now consider
\[
\Phi^{\eps}_{\mu}(x) = x + \eps \int_{\TT^d} \KK(x,y)h(y)dy
\]
for some $\KK \in C^{\infty}(\TT^{2d},\TT^d)$ (e.g. $\KK(x,y)=\kappa(x-y)$ for diffusive coupling.) Then
\[
\partial^{\alpha}(\Phi^{\eps}_{h})_j=
\begin{cases}
1+\eps \int_{\TT^d} \partial^{\alpha} \KK(x,y)h(y)dy  & \quad \text{if } \partial^{\alpha}=\partial_{x_j} \\
\eps \int_{\TT^d} \partial^{\alpha} \KK(x,y)h(y)dy & \quad \text{otherwise}.
\end{cases}
\]
Assumptions \eqref{eq:CouplingAssum1}-\eqref{eq:CouplingAssum2} are checked similarly:
\[
d_{C^{r}}(\Phi_{h}^{\eps},\Phi_{h}^{\eps'})\leq |\varepsilon-\eps'| \left|\int_{\TT^d} \partial^{\alpha} \KK(x,y)h(y)dy\right| \leq  C(\KK)|\eps-\eps'| \|h\|_{TV}.
\]
and 
\[
d_{C^{r}}(\Phi_{h}^{\eps},\Phi_{h'}^{\eps})\leq  |\varepsilon| \left|\int_{\TT^d} \partial^{\alpha} \KK(x,y)(h-h')(y)dy\right| \leq C(\KK)|\varepsilon| \|h-h'\|_{0,r}.
\]
\end{example}
\section{Proofs}\label{sec:proofs}
\subsection{ Transfer operators for sequential Anosov maps}\ \\
The following Lasota-Yorke inequalities hold for the transfer operator $\L_T$:
\begin{proposition} \label{prop:LY-T} \cite[Proposition 3.2, Lemma 4.1]{BL21}
For each $\theta\in (\max\{\nu_0,\lambda_0^{-1}\},1)$, there exist constants $A, B>0$  such that, for all $h\in \B^{0,q}$, $q\in \{0,\dots,r- 1\}$, holds true
\[
\|\L^n_{ T} h\|_{0,q}\leq A \|h\|_{0,q}.
\]
In addition, for $h\in \B^{1,q}$ and all $q\in \{1,\dots,r- 2\}$, holds true
\[
\begin{split}
&\|\L^n_{T} h\|_{0,q}\leq A \theta^n\|h\|_{0,q}+B\|h\|_{0,q+1};\\
&\|\L^n_{T} h\|_{1,q}\leq A\theta^n\|h\|_{1,q}+B \|h\|_{0,q+1}.
\end{split}
\]
Moreover, $\{h\in\B^{1,q}\;:\;\|h\|_{1,q}\leq 1\}$ is relatively compact in the topology associated to the norm $\|h\|_{0,q+1}$.
\end{proposition}
Since we assumed that $T$ is transitive, the above proposition implies that $\L_T$ admits a spectral gap when acting on $\B^{1,q}$, see \cite{BL21} for a detailed discussion.

\begin{remark}\label{rem:seq}
Choose $\eps^* > 0 $ sufficiently small, and let $h_0,\dots,h_{n-1} \in \cM_1(M)$. Each concatenation $T_{h_{n-1}}^{\eps}\circ\cdots\circ T_{h_0}^{\eps}$  is a composition of Anosov diffeomorphisms with properties that can be uniformly controlled for any sequence $h_i$, $i=0,\dots,n-1$, $n \in \mathbb{N}$ and any $|\eps| < \eps^*$ (for an argument see Lemma \ref{lem:anosovcoup} in the Appendix). We will use this information to obtain uniform Lasota-Yorke inequalities for all $\L_{T_{h_{n-1}}^{\eps}\circ\cdots\circ T_{h_0}^{\eps}}$ provided $|\eps|$ is sufficiently small. 
\end{remark}
\begin{proposition} \label{prop:LY-Seq}
There exists $\eps_1 > 0$ such that for each $|\varepsilon| < \varepsilon_1$, $\theta\in (\max\{\nu,\lambda^{-1}\},1)$ and constants $ A,  B>0$ such that for all $h$, $g_0,\dots,g_{n-1}  \in \B_1^{0,q}$, $n \in \NN$, $q\in \{0,\dots,r- 1\}$, holds true
\begin{equation} \label{eq:LY1Th}
\|\L_{T_{g_{n-1}}^{\varepsilon} \circ \cdots \circ {T_{g_0}^{\varepsilon}}} h\|_{0,q}\leq  A\|h\|_{0,q}.
\end{equation}
In addition, for all $h$, $g_0,\dots,g_{n-1}  \in \B^{1,q} \cap \B_1^{0,q}$ and all $q\in \{1,\dots,r- 2\}$, holds true
\begin{equation} \label{eq:LY2Th}
\begin{split}
&\|\L_{T_{g_{n-1}}^{\varepsilon} \circ \cdots \circ {T_{g_0}^{\varepsilon}}} h\|_{0,q}\leq  A  \theta^n\|h\|_{0,q}+ B\|h\|_{0,q+1};\\
&\|\L_{T_{g_{n-1}}^{\varepsilon} \circ \cdots \circ {T_{g_0}^{\varepsilon}}} h\|_{1,q}\leq  A \theta^n\|h\|_{1,q}+ B\|h\|_{0,q+1}.
\end{split}
\end{equation}
\end{proposition}
\begin{proof}
By Remark \ref{re:allanosov}, with small changes in notation\footnote{Basically by replacing $T^n$ in the proofs of \cite{BL21} by $T_{g_{n-1}}^{\varepsilon} \circ \cdots \circ {T_{g_0}^{\varepsilon}}$.} the proof follows verbatim as that of [\cite{BL21} Proposition 3.2].
\end{proof}

The following statement is an essential perturbation lemma relating the action of the transfer operators to the distance of the associated coupled Anosov maps.
\begin{lemma} \label{lem:lip} 	 
Let $h \in \B^{1,q} \cap \B_1^{0,q}$, $q \in \{1,\dots,r-2\}$, $g, g_1, g_2  \in \cM_1(M)$, $|\eps|, |\eps'| < \eps_2$ for some $\eps_2 \leq \eps_1$. Then
	\begin{align}
	\|(\L_{T_{g_1}^{\eps}}-\L_{T_{g_2}^{\eps}})h\|_{0,q+1} &\leq  C \|h\|_{1,q}d_{\C^{r}}(T_{g_1}^{\eps},T_{g_2}^{\eps}); \label{eq:lip1w} \\
\|(\L_{T_{g}^{\eps}}-\L_{T_{g}^{\eps'}})h\|_{0,q+1} &\leq  C \|h\|_{1,q}d_{\C^{r}}(T_{g}^{\eps},T_{g}^{\eps'}).\label{eq:lip2w}
	\end{align}
\end{lemma}

\begin{proof}
Assume $|\eps|$ is small enough, and the domains $V_i$ of the charts are small enough such that there is an open set $U_j$  for any $i$ containing both $T_{g_1}^{\eps} V_i$ and $T_{g_2}^{\eps} V_i$, on which we can define functions $\psi_j \in \C^r$, mapping  $U_j$ to $\RR^d$. Then we can define $T_t = tT_{g_1}^{\eps}+(1-t)T_{g_2}^{\eps}$, $t \in [0,1]$, in the chart $\{U_j, \psi_j\}$. The computation below should be understood in these charts.

According to Lemma \ref{lem:interanosovcoup}, $T_t$ is an Anosov diffeomorphism such that the stable and unstable cones can be chosen uniformly not only in $g_i$ and $\eps$ but also in $t$. We are going to write $\vf_t = \vf \circ T_t$ for a test function $\vf \in \C^q(M,\CC^{\ell})$ and $\FF_t = T_t^{-1}\FF$ for a foliation $\FF$. Notice that the framework of \cite{BL21} applies to this foliation, so in particular for $T_t^{-1}W \in \FF_t$ it holds that $T_t^{-1}W \in {\mathcal{F}}^r_\C$.
Let $h \in \C^1(M,\CC)$ and $\vf \in \C^{q+1}(M,\CC)$. 
\begin{align*}
&\int_M (\L_{T_{g_1}^{\eps}}-\L_{T_{g_2}^{\eps}})h \vf \:d \omega =  \int_M h(\vf \circ T_{g_1}^{\eps} - \vf \circ T_{g_2}^{\eps}) d\omega \\
&= \int_M h \int_0^1\langle \nabla \vf \circ T_t ,(T_{g_2}^{\eps}-T_{g_1}^{\eps})\rangle dt d\omega \\
&= \int_0^1 \left(\int_M  h \divv (\vf \circ T_t (T_{g_2}^{\eps}-T_{g_1}^{\eps})) d\omega  - \int_M  h \vf \circ T_t \tr D(T_{g_2}^{\eps}-T_{g_1}^{\eps})) d\omega\right) dt \\
&\leq \|h\|_{1,q}\int_0^1 \|\vf \circ T_t (T_{g_2}^{\eps}-T_{g_1}^{\eps})\|_{q+1}^{T_t^{-1}W}dt \\
&+ \|h\|_{0,q+1}\int_0^1 \|\vf \circ T_t \tr D(T_{g_2}^{\eps}-T_{g_1}^{\eps})\|_{q+1}^{T_t^{-1}W}dt \\
&:= (I) + (II).
\end{align*}
First consider $\|\vf \circ T_t (T_{g_2}^{\eps}-T_{g_1}^{\eps})\|_{q+1}^{T_t^{-1}W}$ and for $\xi\in M$ write $(\FF_t)_\xi(x,y):=((F_t)_\xi(x,y),y)$ which describes the local foliation\footnote{See subsection \ref{sec:foliation} in the appendix for more details about foliations.}.
\begin{align*}
&\|\vf \circ T_t (T_{g_2}^{\eps}-T_{g_1}^{\eps})\|_{q+1}^{T_t^{-1}W} \\
&=\sup_{\substack{\xi \in M \\ x \in U^0_u}}\sum_{j=1}^d \|[\{\vf \circ T_t (T_{g_2}^{\eps}-T_{g_1}^{\eps})\}\circ (\FF_t)_\xi(x,\cdot)]_j\|_{\C^{q+1}(U_s^0,\CC)} \\
&\leq \sup_{\substack{\xi \in M \\ x \in U^0_u}}\|\vf \circ T_t \circ (\FF_t)_\xi(x,\cdot)\|_{\C^{q+1}(U_s^0,\CC)}\cdot \\
&\quad \times \sup_{\substack{\xi \in M \\ x \in U^0_u}}\sum_{j=1}^d \|[ (T_{g_2}^{\eps}-T_{g_1}^{\eps})\circ (\FF_t)_\xi(x,\cdot)]_j\|_{\C^{q+1}(U_s^0,\CC)} \\
&= \|\vf \circ T_t\|_{q+1}^{T_t^{-1}W} \sup_{\substack{\xi \in M \\ x \in U^0_u}}\sum_{j=1}^d \|[ (T_{g_2}^{\eps}-T_{g_1}^{\eps})\circ (\FF_t)_\xi(x,\cdot)]_j\|_{\C^{q+1}(U_s^0,\CC)} \\
&\leq A_0 \|\vf\|_{q+1}^W \sup_{\substack{\xi \in M \\ x \in U^0_u}}\sum_{j=1}^d \|[ (T_{g_2}^{\eps}-T_{g_1}^{\eps})\circ (\FF_t)_\xi(x,\cdot)]_j\|_{\C^{q+1}(U_s^0,\CC)}
\end{align*}
using \cite[Lemma 2.15]{BL21}. By Lemma \ref{lem:prop-norm} we obtain
\begin{align*}
&\|[ (T_{g_2}^{\eps}-T_{g_1}^{\eps})\circ (\FF_t)_\xi(x,\cdot)]_j\|_{\C^{q+1}(U_s^0,\CC)} \\
&\leq \|[ T_{g_2}^{\eps}-T_{g_1}^{\eps}]_j\|_{\C^{q+1}(U_s^0,\CC)}\sum_{i=0}^{q+1}\binom{q+1}{i}\param^{q+1-i}\|D_y(\FF_t)^T_\xi(x,\cdot)\|^i_{\C^{q+1}(U_s^0,\CC)} \\
&\leq const(r,q,L,\param) \cdot d_{\C^{q+1}}(T_{g_1}^{\eps},T_{g_2}^{\eps})
\end{align*}
This gives
\[
(I) \leq A_0 \|\vf\|_{q+1}^W \cdot const'(r,q,L,\param,d) \cdot d_{\C^{q+1}}(T_{g_1}^{\eps},T_{g_2}^{\eps})\|h\|_{1,q}.
\]
We can do a similar calculation for (II) and obtain 
\begin{align*}
&\|\vf \circ T_t \tr D(T_{g_2}^{\eps}-T_{g_1}^{\eps})\|_{q+1}^{T_t^{-1}W} \\
&\leq A_0 \|\vf\|_{q+1}^W \sup_{\substack{\xi \in M \\ x \in U^0_u}}\sum_{j=1}^d \|[\tr D (T_{g_2}^{\eps}-T_{g_1}^{\eps})\circ (\FF_t)_\xi(x,\cdot)]_j\|_{\C^{q+1}(U_s^0,\CC)} \\
&\leq A_0 \|\vf\|_{q+1}^W\cdot const'(r,q,L,\param,d) \|[ \tr D(T_{g_2}^{\eps}-T_{g_1}^{\eps})]_j\|_{\C^{q+1}(U_s^0,\CC)}
\end{align*}
so 
\[
(II) \leq A_0 { \|\vf\|_{q+1}^W}\cdot const'(r,q,L,\param,d) \cdot d_{\C^{q+2}}(T_{g_1}^{\eps},T_{g_2}^{\eps})\|h\|_{0,q+1}.
\]
We obtained
\[
\left| \int_M (\L_{T_{g_1}^{\eps}}-\L_{T_{g_2}^{\eps}})h \vf \:d \omega \right| \leq C\|\vf\|_{q+1}^W\|h\|_{1,q}d_{\C^{q+2}}(T_{g_1}^{\eps},T_{g_2}^{\eps}),
\]
and the fact that $\C^1(M,\CC)$ is dense in $\B^{0,q+1}$ concludes the proof \eqref{eq:lip1w}. 

For \eqref{eq:lip2w}, an analogous argument works. We can prove an analogue of Lemma \ref{lem:interanosovcoup} for $T_t=tT_g^{\eps}+(1-t)T_g^{\eps'}$ (provided that $|\eps|, |\eps'| < \eps^*$ of Lemma \ref{lem:anosovcoup}) and repeat the above argument for $T_g^{\eps}, T_g^{\eps'}$ instead of $T_{g_1}^{\eps},T_{g_2}^{\eps}$.

\end{proof}
Using the assumptions on the coupling, we obtain the following corollary:
\begin{corollary} \label{cor:statstab}
By Lemma \ref{lem:lip} and Assumption \eqref{eq:CouplingAssum1} it follows
\begin{equation} \label{eq:pert1}
	\|(\L_{T_{g_1}^{\eps}}-\L_{T_{g_2}^{\eps}})h\|_{0,q+1} \leq  C|\varepsilon| \|g_1-g_2\|_{0,q+1} \|h\|_{1,q}.
\end{equation}
and by assumption \eqref{eq:CouplingAssum2},
\begin{equation} \label{eq:pert2}
	 \|(\L_{T_{g}^{\eps}}-\L_{T_{g}^{\eps'}})h\|_{0,q+1} \leq  C|\varepsilon-\eps'| \|g\|_{TV} \|h\|_{1,q}.
\end{equation}
The above implies
\begin{align*}
&\|(\L_{T_{g_n}^{\eps}}\cdots \L_{T_{g_1}^{\eps}}-\L_{T_{f_n}^{\eps}}\cdots \L_{T_{f_1}^{\eps}})h\|_{0,q+1} \\
& \leq \sum_{i=1}^{n}\|\L_{T_{g_n}^{\eps}}\dots \L_{T_{g_{i+1}}^{\eps}}(\L_{T_{g_i}^{\eps}}- \L_{T_{f_i}^{\eps}})\L_{T_{f_{i-1}}^{\eps}}\dots\L_{T_{f_1}^{\eps}}h\|_{0,q+1} \\
& \leq C(n)|\eps|\max_i  \|g_i-f_i\|_{0,q+1} \|h\|_{1,q},
\end{align*}
and in particular for $g_i,f_i \in \overline{\B}(K,q)$,
\begin{equation} \label{eq:iterss}
\|(\L_{T_{g_n}^{\eps}}\cdots \L_{T_{g_1}^{\eps}}-\L_{T_{f_n}^{\eps}}\cdots \L_{T_{f_1}^{\eps}})h\|_{0,q+1} \leq C(n,K)|\eps|\max_i\|g_i-f_i\|_{0,q+1}
\|h\|_{1,q},
\end{equation}
and by a similar argument
\begin{equation} \label{eq:iterss2}
\|(\L_{T_{g_n}^{\eps}}\cdots \L_{T_{g_1}^{\eps}}-\L_{T_{g_n}^{\eps'}}\cdots \L_{T_{g_1}^{\eps'}})h\|_{0,q+1} \leq C(n,K)|\eps-\eps'|\max_i\|g_i\|_{TV}
\|h\|_{1,q}.
\end{equation}
\end{corollary}

Next, we prove memory loss, which will also be needed for the proof of exponential convergence to equilibrium in the next subsection. \begin{lemma} \label{lem:Contraction} 
	Let $K \geq K_{\min}$. There exists $\eps_3 > 0$, $C > 0$ such that for all $|\eps| < \eps_3$, $q \in \{1,\dots,r-2\}$
		\[
	 \|\L_{T_{g_{n}}^\eps}\dots \L_{T_{g_1}^\eps}h\|_{1,q}\leq C \theta^n\|h\|_{1,q}  \quad n \in \NN
	\]
	holds true for all $h\in\B^{1,q}$, $h(1)=0$ and all $g_1,\dots,g_{n} \in \cM_1(M)$.
\end{lemma}
\begin{proof}
By Proposition \ref{prop:LY-Seq}, Corollary \ref{cor:statstab} (in particular \eqref{eq:iterss2}) and the fact that $\L_T$ admits a spectral gap on $\B^{1,q}$, we have

\begin{equation} \label{eq:LY2scit}
\begin{split}
&\|\L_{T_{g_{n+m}}^\eps}\dots \L_{T_{g_1}^\eps}h\|_{1,q} \le  A \theta^n\|\L_{T_{g_{m}}^\eps}\dots \L_{T_{g_1}^\eps}h\|_{1,q}+B\|\L_{T_{g_{n}}^\eps}\dots \L_{T_{g_1}^\eps}h\|_{0,q+1}\\
&\le  A \theta^n\|\L_{T_{g_{m}}^\eps}\dots \L_{T_{g_1}^\eps}h\|_{1,q}+B\|(\L^m_T-\L_{T_{g_{m}}^\eps}\dots \L_{T_{g_1}^\eps})h\|_{0,q+1}+B\|\L^m_Th\|_{0,q+1}\\
&\phantom{\le} +B\|\L^m_Th\|_{0,q+1}\\
&\le A_1 \theta^n\|h\|_{1,q}+ B(m,K)|\eps| \|h\|_{1,q}+    B \|\L^m_Th\|_{0,q+1}\\
&\le\left( A_1 \theta^n+B(m,K)|\eps|+    B_1\sigma^m \right)\|h\|_{1,q},
\end{split}
\end{equation}
for some constants $B(m,K),B_1>0$ and $\sigma\in(0,1)$, since $h(1)=0$. Choose $n$ so that $ A_1 \theta_0^n\le\frac{\theta}{3}$ for some $\theta\in(0,1)$. Then choose $m$ so that $B_1\sigma^m\le \frac{\theta}{3}$. Finally, choose $|\eps|$ small enough so that $B(m,K)|\eps|\le\frac{\theta}{3}$. Therefore,
$$ \|\L_{T_{g_{n+m}}^\eps}\dots \L_{T_{g_1}^\eps} h\|_{1,q}\le \theta\|h\|_{1,q}.$$
This implies
\[
	\|\L_{T_{g_{n}}^\eps}\dots \L_{T_{g_1}^\eps}h\|_{1,q}\leq C \theta^n\|h\|_{1,q}  \quad\text{for all }n \in \NN.
	\]
\end{proof}

\subsection{The invariance of $\overline{\B}(K,q)$}\label{sec:Leps is well defined}\ \\
We show that $\L_\eps$ is well defined and that $\overline{\B}(K,q)$, defined in \eqref{eq:bk}, is invariant. 
\begin{proof}[\bf Proof of Proposition \ref{lem:invariance}] 
Let $h\in \C^1(M)$. Note that for $h\ge 0$, $\L_\eps(h)\ge 0$ and $\int\L_\eps(h)=\int h$. Consequently,  $\L_\eps(h)\in \cM_1(M)\cap \C^1(M)$. The first invariance results follow then by closing with respect to the $\|\cdot\|_{q+1}$ norm and recalling equation \eqref{eq:LY1Th}. 

Next, let $h\in \B(K,q)$. By \eqref{eq:LY1Th} we have
\begin{equation} \label{eq:LY1sc}
\|\L_{\eps}^n (h)\|_{0,q}\leq  A\|h\|_{0,q}
\end{equation}
and by \eqref{eq:LY2Th}
\begin{equation} \label{eq:LY2sc}
\begin{split}
&\|\L_{\eps}^n (h)\|_{0,q}\leq  A  \theta^n\|h\|_{0,q}+ B\|h\|_{0,q+1};\\
&\|\L_{\eps}^n (h)\|_{1,q}\leq  A \theta^n\|h\|_{1,q}+ B\|h\|_{0,q+1}.
\end{split}
\end{equation}
Thus, we have
\begin{align}\label{eq:little_growth}
\|\L_{\eps}^n (h)\|_{1,q}
& \leq A\theta^n K + B \|h\|_{0,q+1}.
\end{align}
Choose $N$ large enough, such that $A\theta^N=\beta\in(0,1)$. By \cite[Remark 2.15]{BL21} we have $\| h\|_{0,q+1} \leq \| h\|_{1}$ (where $\| h\|_{1}=\int_M |h|$) and since we work with nonnegative distributions we have for each $n\geq N$, 
\begin{equation}\label{eq:invariant}
\begin{split}
\|\L_{\eps}^n (h)\|_{1,q} &\leq \theta^n AK +  B \int h\\
&= \beta K+  B.
\end{split}
\end{equation}
Choose $K$ such that $K \geq K_{\min}:=\frac{B}{1-\beta}$ holds. This completes the proof of the proposition. 
\end{proof}

We are now ready to make a statment about the fixed points of $\L^N_\eps$.
\begin{proposition} \label{prop:uniqueN}
There exists $N \in \NN$, $K_{\min} \ge 0$, and $\eps_4 > 0$ such that $\L_{\eps}^N$ has a fixed point in $\overline{\B}(K,q)$ for all $K \ge K_{\min}$ and all $|\eps| < \eps_4$.
\end{proposition}
\begin{proof}
Let $h_1,h_2 \in \overline{\B}(K,q)$. Note that, by equation \eqref{eq:little_growth} we have 
\begin{equation}\label{eq:contractw0}
\|\L_\eps^n h_i\|_{1,q}\leq AK+B.
\end{equation}
Then, by Proposition \ref{prop:LY-Seq} and Equation \eqref{eq:pert1} we have, for all $\tilde h_i\in \B(AK+B,q)$,
\[
\begin{split}
\|\L_\eps \tilde h_1-\L_\eps \tilde h_2\|_{0,q+1}&\leq \|\L_{T^\eps_{\tilde h_1}} \tilde h_1-\L_{T^\eps_{\tilde h_1}} \tilde h_2\|_{0,q+1}+\|\L_{T^\eps_{\tilde h_1}} \tilde h_2-\L_{T^\eps_{\tilde h_2}}\tilde h_2\|_{0,q+1}\\
&\leq (A+\eps C [AK+B])\|\tilde h_1-\tilde h_2\|_{0,q+1}.
\end{split}
\]
Hence, $\L^N_{\eps}|_{\overline{\B}(K,q)}$ is continuous in the weak norm $\|\cdot\|_{0,q+1}$, being the composition of continuous operators. Moreover, we have shown in Proposition \ref{lem:invariance} that $\overline{\B}(K,q)$ is invariant under the action of $\L_\eps^N$. Since $\overline{\B}(K,q)$ is a convex, compact metric space, we obtain that $\L^N_{\eps}$ has at least one fixed point in $\overline{\B}(K,q)$. 
\end{proof}
\begin{remark} Note that the Proposition \ref{prop:uniqueN} does not say much on the dynamics of $\L_\ve$, just the existence of  periodic orbits in $\overline{\B}(K,q)$. To know if $\L_\eps$ has a unique fixed point, some more work is needed.
\end{remark}

\subsection{ Sequential Anosov maps and a stronger norm}\ \\
In the following lemma we prove a Lasota-Yorke inequality for Anosov diffeomorphisms, for the Banach space $\B^{2,q}$ (see \eqref{eq:highernorms} to recall the definition of this space). In particular, Lemma \ref{lem:SLY} below can be useful outside the scope of this paper. We note however, that we do not prove that the unit ball of $\B^{2,q}$ is compactly embedded in $\B^{1,q+1}$ since the latter is not needed for the current work\footnote{We warn the reader that to obtain information about the spectral properties of $\L_T$ when acting on $\B^{2,q}$, a compact embedding result is needed.}.
\begin{lemma}\label{lem:SLY}
There exists $\theta\in (\max\{\nu,\lambda^{-1}\},1)$ and constants $ A_1,  A_2>0$ such that for all $h \in \B^{2,q}$, $n \in \NN$, $q\in \{1,\dots,r- 2\}$, holds true
$$\|\mathcal L^{n}_T h\|_{2,q}\le A_1\theta^{n} \|h\|_{2,q}+A_2\|h\|_{1,q+1}.$$
\end{lemma}
\begin{proof}
All the operations in this proof are understood in the charts introduced in subsection \ref{subsec:map}. We have
\begin{equation}\label{eq:diftran}
\begin{split}
&\partial_{x_i}\mathcal{L}_T ^nh 
=\partial_{x_i} \left( \frac{h}{|\det DT^n|} \circ T^{-n}\right) \\
&=\sum_{j=1}^d \partial_{x_j} \left(\frac{ h}{|\det DT^n|}\right) \circ T^{-n} \cdot \{DT^{-n}\}_{ji} \\
&=\sum_{j=1}^d \frac{\partial_{x_j} h}{|\det DT^n|} \circ T^{-n} \cdot \{DT^{-n}\}_{ji}\\
&\phantom{=}-\sum_{j=1}^d \frac{h\partial_{x_j}|\det DT^n|}{|\det DT^n|^2} \circ T^{-n} \cdot \{DT^{-n}\}_{ji} \\
&=\sum_{j=1}^d \mathcal{L}_T^n(\partial_{x_j}h) \cdot \{DT^{-n}\}_{ji} - \sum_{j=1}^d \mathcal{L}_T^n\left(\frac{h\partial_{x_j}|\det DT^n|}{|\det DT^n|}\right) \cdot \{DT^{-n}\}_{ji},
\end{split}
\end{equation}
Accordingly, letting $(W,\vf^i)\in \Omega_{L,q+2, d}$ and using \eqref{eq:diftran},  we can write
\[
\begin{split}
&\sum_{i=1}^d \int\partial_{x_i}\mathcal{L}_T^n h \divv\vf^i\\
&=\sum_{i=1}^d \sum_{j=1}^d\int (\partial_{x_j}h) \cdot\left( \{DT^{-n}\}_{ji} \divv\vf^{i}\right)\circ T^n \\
&\phantom{=}
- \sum_{i=1}^d \sum_{j=1}^d\int \left(\frac{h\partial_{x_j}|\det DT^n|}{|\det DT^n|}\right) \cdot \left(\{DT^{-n}\}_{ji} \divv\vf^{i}\right)\circ T^n\\
&=\sum_{j=1}^d\int (\partial_{x_j}h) \divv( \tilde\vf^j)\circ T^n-\sum_{j=1}^d\int (\partial_{x_j}h)\sum_{i=1}^d\sum_{l=1}^d \vf_l^i \circ T^n \partial_{x_l}\left[ \{DT^{-n}\}_{ji}\circ T^n \right]\\
&-\sum_{j=1}^d\int \left(\frac{h\partial_{x_j}|\det DT^n|}{|\det DT^n|}\right) \cdot (\divv\tilde \vf^j)\circ T^n\\
&\phantom{=}
+\sum_{j=1}^d\int \left(\frac{h\partial_{x_j}|\det DT^n|}{|\det DT^n|}\right) \cdot \sum_{i=1}^d\sum_{l=1}^d \vf_l^i \circ T^n \partial_{x_l}\left[ \{DT^{-n}\}_{ji}\circ T^n \right],
\end{split}
\]
where we have used the notation $ \sum_{i=1}^d\{DT^{-n}\}_{ji}\vf^i:= \tilde\vf^j$.
To continue, we need the following fact: for $h,f,\vf \in \C^r$, 
\begin{equation}\label{eq:divsplit}
\begin{split}
&\int hf \divv \vf = \int h \divv(f\vf) - \int h \sum_{l=1}^d \vf_l \partial_{x_l}f\\
& (\divv \vf)\circ T^n=\divv ((DT^{n})^{-1}\vf\circ T^n)-\sum_{i=1}^d \partial_{x_i}(DT^n)^{-1}_{ij}\vf_{j}\circ T.
\end{split}
\end{equation}
Using \eqref{eq:divsplit} and integrating by part, we can write the above as
\begin{equation}\label{eq:semiess}
\begin{split}
\sum_{i=1}^d \int\partial_{x_i}\mathcal{L}_T^n h \divv\vf^i=&
\sum_{j=1}^d \int (\partial_{x_j}h) \cdot \divv\left(  (DT^{n})^{-1}\tilde \vf^{j}\circ T^n \right)\\
+\int h\divv \Theta^1+\int h\Theta^0.
\end{split}
\end{equation}
where $\|\Theta^1\|^{T^{-n}W}_{q+1}+\|\Theta^0\|^{T^{-n}W}_{q+1}\leq C_n$. This follows from the fact that $T\in\cC^r$ and \cite[Lemma 2.15]{BL21}.

It remains to control the term with two derivatives.
To this end we use the projectors $\pi^u,\pi^s$ similarly to what is done in \cite{BL21}. See appendix \ref{sec:projection} for a precise defintion and their properties.

We can thus write
\begin{equation} \label{eq:fourterms}
\begin{split}
&\sum_{j=1}^d \int (\partial_{x_j}h) \cdot \divv\left(  (DT^{n})^{-1}\tilde \vf^{j}\circ T^n \right)=\\
=&-\sum_{j,i,k,l,t,t'=1}^d\int h\partial_{x_j} \partial_{x_i}\left\{(DT^n)^{-1}_{it}\pi^u_{tl}\circ T^n
(DT^n)^{-1}_{jt'}\pi^u_{t'k}\circ T^n \vf^k_l\circ T^n\right\}\\
&-\sum_{j,i,k,l,t,t'=1}^d\int h\partial_{x_j} \partial_{x_i}\left\{(DT^n)^{-1}_{it}\pi^u_{tl}\circ T^n
(DT^n)^{-1}_{jt'}\pi^s_{t'k}\circ T^n \vf^k_l\circ T^n\right\}\\
&-\sum_{j,i,k,l,t,t'=1}^d\int h\partial_{x_j} \partial_{x_i}\left\{(DT^n)^{-1}_{it}\pi^s_{tl}\circ T^n
(DT^n)^{-1}_{jt'}\pi^u_{t'k}\circ T^n \vf^k_l\circ T^n\right\}\\
&-\sum_{j,i,k,l,t,t'=1}^d\int h\partial_{x_j} \partial_{x_i}\left\{(DT^n)^{-1}_{it}\pi^s_{tl}\circ T^n
(DT^n)^{-1}_{jt'}\pi^s_{t'k}\circ T^n \vf^k_l\circ T^n\right\}.
\end{split}
\end{equation}
Let us analyze the above terms one by one. If we set 
\[
\boeta^j_{l}=\sum_{t',k=1}^{d}(DT^n)^{-1}_{jt'}\circ T^{-n}\pi^u_{t'k} \vf^k_l,\\
\]
since $\pi^u\boeta^j$ belongs to the unstable cone, using \eqref{eq:piu} twice we have, for the estimate of the first term
\[
\begin{split}
\|(DT^{-n})^{-1}\circ T^{-n}\pi^u\boeta^{j}\|^{W}_{q+2}&\leq C\lambda^{-n}\|\boeta^j\|^W_{q+2}+\frac{C_n}{\varpi}\| \boeta^{j}\|^{W}_{q+1}\\
&\leq C\lambda^{-2n}\|\bar\vf\|^W_{q+2}+\frac{C_n}{\bar\varpi}\| \bar\vf\|^{W}_{q+1},
\end{split}
\]
where $\|\bar\vf\|^W_{q'}=\sup_{i,j}\|\vf_i^j\|^W_{q'}$ for each $q'\leq r$.
Consequently, by \cite[Lemma 2.11]{BL21}, we obtain
\begin{equation}\label{eq:annoy1}
\|(DT^{-n})^{-1}(\pi^u\boeta^j)\circ T^n\|^{T^{-n}W}_{q+2}\leq C\lambda^{-2n}\|\bar \vf\|^W_{q+2}+\frac{C_n}{\varpi}\| \bar \vf\|^{W}_{q+1}.
\end{equation}
Note that the second and third terms in \eqref{eq:fourterms} are essentially the same. Indeed, exchanging $i$ and $j$ in the second yields the third apart from the fact that $\vf^k_l$ is substituted by $\vf^l_k$, which is irrelevant. We can thus analyze only the third term.
By \eqref{eq:pis} and \eqref{eq:piu}, we have
\begin{equation}\label{eq:annoy2}
\begin{split}
&\|\sum_{i=1}^d\partial_{x_i}[(DT^{-n})^{-1}\pi^s\circ T^{n}\boeta^j\circ T^{n}]_i\|^{T^{-n}W}_{q+1}\leq C_n\|\boeta^j\|^W_{q+2}\\
&\phantom{\|\sum_{i=1}^d\partial_{x_i}[(DT^{-n})^{-1}\pi^s\circ T^{n}}
\leq C_n\lambda^{-n}\|\bar \vf\|^W_{q+2}+\frac{C_n}{\varpi}\|\bar \vf\|^W_{q+1}.
\end{split}
\end{equation}
To treat the last term in \eqref{eq:fourterms}, let
\[
\btheta^j_ {l}=\sum_{t',k=1}^{d}(DT^n)^{-1}_{jt'}\circ T^{-n}\pi^s_{t'k} \vf^k_l,
\]
we can then write it as
\[
\sum_{j=1}^d\int h\partial_{x_j}\left\{\sum_{i,t,l=1}^d \partial_{x_i}(DT^n)^{-1}_{it}\pi^s_{tl}\circ T^n
\btheta^j_ {l}\circ T^n\right\}
\]
Then, by \eqref{eq:pis} and \eqref{eq:pis2} we have
\[
\|\sum_{i,t,l=1}^d \partial_{x_i}(DT^n)^{-1}_{is}\pi^s_{tl}\circ T^n
\btheta^j_ {l}\|^{T^{-n}W}_{q+1}\leq C_n\sup_{l,j}\|\btheta^j_ {l}\|^W_{q+2}\leq C_{n,\varpi}\|\bar \vf\|^W_{q+2}.
\]
The above implies that the last term is bounded by
\[
\sum_{j=1}^d\int h\partial_{x_j}\left\{\sum_{i,t,l=1}^d \partial_{x_i}(DT^n)^{-1}_{it}\pi^s_{tl}\circ T^n
\btheta^j_ {l}\circ T^n\right\}\leq C_{n,\varpi}\|h\|^*_{1,q+1}.
\]
 By choosing $\varpi$ large enough in \eqref{eq:annoy1}, using \eqref{eq:annoy2} and \eqref{eq:semiess}, we get
\[
\|\mathcal L^{n}_T h\|^*_{2,q}\le C\lambda^{-2n} \|h\|_{2,q}^*+C_n\|h\|_{1,q+1}.
\]
The above equation can be iterated with steps $n_0$ such that $C\lambda^{-2n_0}\leq \theta<1$.
Now choosing $b$ in the definition of \eqref{eq:highernorms} large enough (depending on the fixed $n_0$) and using the Lasota-Yorke inequality in the $\|\cdot\|_{1,q}$ from \cite{BL21}, we get for all $n\in\mathbb N$
$$\|\mathcal L^{n}_T h\|_{2,q}\le A_1\theta^{n} \|h\|_{2,q}+A_2\|h\|_{1,q+1}.$$
 \end{proof}
Using Remark \ref{rem:seq} and Lemma \ref{lem:SLY}, we obtain the following corollary:
 
\begin{corollary}\label{cor:SLYconsist}
Let $\eps_1 > 0$ be the same as in Proposition \ref{prop:LY-Seq}. For each $|\varepsilon| < \varepsilon_1$ and $\theta\in (\max\{\nu,\lambda^{-1}\},1)$ there exist constants $ A_1,  A_2>0$ such that for all $h$, $g_0,\dots,g_{n-1}  \in \B^{2,q}\cap \B_1^{0,q}$, $n \in \NN$, $q\in \{1,\dots,r- 2\}$, holds true
\begin{equation} \label{eq:SLY1sc}
\|\L_{T_{g_{n-1}}^{\varepsilon} \circ \cdots \circ {T_{g_0}^{\varepsilon}}} h\|_{2,q}\leq  A \theta^n\|h\|_{2,q}+ B\|h\|_{1,q+1};
\end{equation}
that implies
$$
\|\mathcal L_\eps^{n} (h)\|_{2,q}\le A_1\theta^{n} \|h\|_{2,q}+A_2\|h\|_{1,q+1}.$$
\end{corollary}
\begin{lemma} \label{lem:SDifference}
For $h\in \B^{2,q}\cap \B_1^{0,q}$, $q \in \{1,\dots,r- 2\}$ we have
\[
\|(\L_{T_{g_1}^{\eps}}-\L_{T_{g_2}^{\eps}})h\|_{1,q+1} \leq  Cd_{C^{r}}(T_{g_1}^{\eps},T_{g_2}^{\eps} )\|h\|_{2,q};
\]
\[
\|(\L_{T_{g}^{\eps}}-\L_{T_{g}^{\eps'}})h\|_{1,q+1} \leq  Cd_{C^{r}}(T_{g}^{\eps},T_{g}^{\eps'} )\|h\|_{2,q}.
\]

\end{lemma}
\begin{proof}
Recall the definition of $T_t$ in the proof of Lemma \ref{lem:lip}. The calculation below is understood in the charts as explained at the beginning of the proof of Lemma \ref{lem:lip}. Since
 $$\|(\L_{T_{g_1}^{\eps}}-\L_{T_{g_2}^{\eps}})h\|_{1,q+1}= a\|(\L_{T_{g_1}^{\eps}}-\L_{T_{g_2}^{\eps}})h\|_{0,q+1}+\|(\L_{T_{g_1}^{\eps}}-\L_{T_{g_2}^{\eps}})h\|^*_{1,q+1},$$
it is enough to estimate the latter since the former is covered by Lemma \ref{lem:lip}. We have
\begin{equation}\label{eq:I&II}
\begin{split}
&\int_M\left( \L_{T_{g_1}^{\eps}}-\L_{T_{g_2}^{\eps}} \right)h \, \divv \vf\\
&=\int_M h\divv \left(D(T_{g_1}^{\eps})^{-1} \cdot \vf\circ T_{g_1}^{\eps}\right)-\int_M h\divv \left(D(T_{g_2}^{\eps})^{-1}\cdot \vf\circ T_{g_2}^{\eps}\right)\\
&+\int_Mh\sum_{l,k=1}^d \partial_{l} \left[D(T_{g_2}^{\eps})^{-1}\right]_{lk}\cdot \vf_{k}\circ T_{g_2}^{\eps}-\int_Mh\sum_{l,k=1}^d \partial_{l} \left[D(T_{g_1}^{\eps})^{-1}\right]_{lk}\cdot \vf_{k}\circ T_{g_1}^{\eps}\\
&:= (I) +(II).
\end{split}
\end{equation}
We first give a bound on $(I)$. 
\begin{equation}
\begin{split}
&\int_M h\divv \left(D(T_{g_1}^{\eps})^{-1} \cdot \vf\circ T_{g_1}^{\eps}\right)-\int_M h\divv \left(D(T_{g_2}^{\eps})^{-1}\cdot \vf\circ T_{g_2}^{\eps}\right)\\
&=\int_M h\divv \left([D(T_{g_1}^{\eps})^{-1}\circ(T_{g_1}^{\eps})^{-1} \cdot \vf]\circ T_{g_1}^{\eps}-[D(T_{g_2}^{\eps})^{-1}\circ(T_{g_2}^{\eps})^{-1} \cdot \vf]\circ T_{g_2}^{\eps}\right)\\
&=\int_M h\divv \left([D(T_{g_1}^{\eps})^{-1}\circ(T_{g_1}^{\eps})^{-1} \cdot \vf]\circ T_{g_1}^{\eps}-[D(T_{g_1}^{\eps})^{-1}\circ(T_{g_1}^{\eps})^{-1} \cdot \vf]\circ T_{g_2}^{\eps}\right)\\
&+\int_M h\divv \left(\left[D(T_{g_1}^{\eps})^{-1}\circ(T_{g_1}^{\eps})^{-1} -D(T_{g_2}^{\eps})^{-1}\circ(T_{g_2}^{\eps})^{-1}\right]\cdot\vf\circ T_{g_2}^{\eps}\right)\\
&:=(I.I)+(I.II).
\end{split}
\end{equation}

We start with $(I.I)$. Let $D(T_{g_1}^{\eps})^{-1}\circ(T_{g_1}^{\eps})^{-1}\cdot\vf=\hat\vf$  and $T_t = t \cdot T_{g_1}^{\eps}+(1-t)\cdot T_{g_2}^{\eps}, t \in [0,1]$. We have
\begin{equation}
\begin{split}
(I.I)&=-\int_M \nabla h \cdot \left[\hat\vf\circ T_{g_1}^{\eps}-\hat\vf\circ T_{g_2}^{\eps}\right]\\
&=-\sum_{i=1}^d\int_M \partial_{x_i}  h \int_0^1 \nabla \hat \vf_i \circ T_t (T_{g_2}^{\eps}-T_{g_1}^{\eps})dt d\omega \\
&= -\sum_{i=1}^d\int_0^1 \int_M  \partial_{x_i} h \divv (\hat \vf_i \circ T_t ( T_{g_2}^{\eps}-T_{g_1}^{\eps})) d\omega dt  \\
& +\sum_{i=1}^d\int_0^1 \int_M  \partial_{x_i}  h  \hat\vf_i\circ T_t \tr D(T_{g_2}^{\eps}-T_{g_1}^{\eps})) d\omega dt \\
&= -\sum_{i=1}^d\int_0^1 \int_M  \partial_{x_i} h \divv (\hat \vf_i \circ T_t ( T_{g_2}^{\eps}-T_{g_1}^{\eps})) d\omega dt \\
&- \int_0^1\int_M  h \divv\left( \hat\vf\circ T_t \tr D(T_{g_2}^{\eps}-T_{g_1}^{\eps}) \right) d\omega dt \\
&\leq \|h\|_{2,q}\max_i\int_0^1 \|\hat\vf_i\circ T_t (T_{g_2}^{\eps}-T_{g_1}^{\eps})\|_{q+2}^{T_t^{-1}W}dt \\
&+ \|h\|_{1,q+1}\int_0^1 \|\hat\vf\circ T_t \tr D(T_{g_2}^{\eps}-T_{g_1}^{\eps})\|_{q+2}^{T_t^{-1}W}dt
\end{split}
\end{equation}
Using what we obtained in the course of the proof of Lemma \ref{lem:lip}, we get
\begin{equation} \label{eq:boundI.I}
(I.I) \leq C d_{C^{q+2}}(T_{g_1}^{\eps},T_{g_2}^{\eps} )\|h\|_{2,q}.
\end{equation}

As for $(I.II)$, using Lemma \ref{lem:lip1}, we have
\begin{equation}
\begin{split}
&\|\left[D(T_{g_1}^{\eps})^{-1}\circ(T_{g_1}^{\eps})^{-1} -D(T_{g_2}^{\eps})^{-1}\circ(T_{g_2}^{\eps})^{-1} \right]\cdot\vf\circ T_{g_2}^{\eps}\|_{q+1}^W\\
&\le\|\left[D(T_{g_1}^{\eps})^{-1}\circ(T_{g_1}^{\eps})^{-1} -D(T_{g_2}^{\eps})^{-1}\circ(T_{g_2}^{\eps})^{-1}\right]\|_{\C^{q+1}}\cdot\|\vf\circ T_{g_2}^{\eps}\|_{q+1}^W\\
&\le Cd_{C^r}(T^\eps_{g_1},T^\eps_{g_2}) \|\vf\|_{q+1}^W.
\end{split}
\end{equation}
This implies 
\begin{equation} \label{eq:boundI.II} 
(I.II) \leq  C d_{C^{r}}(T_{g_1}^{\eps},T_{g_2}^{\eps} )\|h\|_{1,q+1}.
\end{equation}
Combining \eqref{eq:boundI.I} and \eqref{eq:boundI.II} we obtain
\begin{equation} \label{eq:boundI}
(I) \leq C d_{C^r}(T^\eps_{g_1},T^\eps_{g_2})\|h\|_{2,q}.
\end{equation}
We now bound $(II)$.
\begin{equation}\label{eq:II}
\begin{split}
(II)&= \sum_{l,k=1}^d\int_Mh \partial_{l} \left[D(T_{g_2}^{\eps})^{-1}\right]_{lk}\left(\vf_k\circ T_{g_2}^{\eps}-\vf_k\circ T_{g_1}^{\eps}\right)\\
&+ \sum_{l,k=1}^d \int_Mh\left(\partial_{l} \left[D(T_{g_2}^{\eps})^{-1}\right]_{lk} - \partial_{l} \left[D(T_{g_1}^{\eps})^{-1}\right]_{lk} \right)\vf_k\circ T_{g_1}^{\eps}\\
&:= (II.I) + (II.II).
\end{split}
\end{equation}
For $(II.I)$, by Lemma \ref{lem:lip}, we have
\begin{align}
(II.I)&= \sum_{l,k=1}^d \int_M (\L_{T^\eps_{g_2}}-\L_{T^\eps_{g_1}}) \left(h\partial_{l} \left[D(T_{g_2}^{\eps})^{-1}\right]_{lk} \right)\cdot\vf_k \nonumber\\
&\le C\left\|(\L_{T^\eps_{g_2}}-\L_{T^\eps_{g_1}}) \left(h\partial_{l}\left[D(T_{g_2}^{\eps})^{-1}\right]_{lk} \right)\right\|_{0,q+1}  \label{eq:boundII.I}\\
&\le C d_{C^r}(T^\eps_{g_1},T^\eps_{g_2})\|h\|_{1,q} \nonumber. 
\end{align}
For $(II.II)$, by Corollary \ref{cor:lip} we have 
\begin{equation} \label{eq:boundII.II}
\begin{split}
&\sum_{l,k=1}^d \int_M\L_{T^\eps_{g_1}}h\left(\partial_{l} \left[D(T_{g_2}^{\eps})^{-1}\right]_{lk} \circ (T_{g_1}^{\eps})^{-1} - \partial_{l} \left[D(T_{g_1}^{\eps})^{-1}\right]_{lk}\circ ( T_{g_1}^{\eps})^{-1} \right)\vf_k \\
&\leq \|\L_{T^\eps_{g_1}}h\|_{0,q+1}\sum_{l,k=1}^d\left\|\partial_{l} \left[D(T_{g_2}^{\eps})^{-1}\right]_{lk} \circ (T_{g_1}^{\eps})^{-1} - \partial_{l} \left[D(T_{g_1}^{\eps})^{-1}\right]_{lk}\circ ( T_{g_1}^{\eps})^{-1}\right\|_{\C^{q+1}} \\
&\leq ACd_{\C^r}(T_{g_1}^{\eps}, T_{g_2}^{\eps})\cdot \|h\|_{0,q+1}.
\end{split}
\end{equation}
Using \eqref{eq:boundI}, \eqref{eq:boundII.I} and \eqref{eq:boundII.II} completes the proof of the first inequality of the lemma. The proof for the second part of the lemma follows similarly.
\end{proof}
\begin{corollary} \label{cor:StrongStatStab}
According to \eqref{eq:CouplingAssum1}, we can write
\begin{equation} \label{eq:StrongStatStab}
	\|(\L_{T_{g_1}^{\eps}}-\L_{T_{g_2}^{\eps}})h\|_{1,q+1} \leq  C|\varepsilon| \|g_1-g_2\|_{1,q+1} \|h\|_{2,q},
\end{equation}
and similarly, according to \eqref{eq:CouplingAssum2}
\begin{equation} \label{eq:StrongStatStab1}
	\|(\L_{T_{g}^{\eps}}-\L_{T_{g}^{\eps'}})h\|_{1,q+1} \leq  C|\varepsilon-\varepsilon '| \|g\|_{1,q+1}\|h\|_{2,q}.
\end{equation}
Furthermore,
\begin{align}
&\|(\L_{T_{g_N}^{\eps}}\dots \L_{T_{g_1}^{\eps}}-\L_{T_{f_N}^{\eps}}\dots \L_{T_{f_1}^{\eps}})h\|_{1,q+1} \nonumber\\
& \leq \sum_{i=1}^{N}\|\L_{T_{g_N}^{\eps}}\dots \L_{T_{g_{i+1}}^{\eps}}(\L_{T_{g_i}^{\eps}}- \L_{T_{f_i}^{\eps}})\L_{T_{f_{i-1}}^{\eps}}\dots\L_{T_{f_1}^{\eps}}h\|_{1,q+1} \nonumber\\
& \leq C|\eps|\sum_{i=1}^N \theta^{N-i} \|g_i-f_i\|_{1,q+1} \|h\|_{2,q}. \label{eq:StrongStatStabIter}
\end{align}
where $\theta$ is the contraction factor from Lemma \ref{lem:Contraction}. Similarly,
\begin{align}
\|(\L_{T_{g_N}^{\eps}}\dots \L_{T_{g_1}^{\eps}}-\L_{T_{g_N}^{\eps'}}\dots \L_{T_{g_1}^{\eps'}})h\|_{1,q+1} \leq C \sum_{i=1}^N \theta^{N-i}  \|g_i\|_{1,q+1}|\eps-\eps'| \|h\|_{2,q}. \label{eq:StrongStatStabIterEps}
\end{align}
\end{corollary}
\subsection{ Fixed point uniqueness and exponential convergence in $\B(K_1,K_2,q)$}\label{sec:unique}\ \\
Recall the definition of the set $\B(K_1,K_2,q)$ (see  \eqref{eq:bk2}).
\begin{lemma}\label{lem:Sinvariance}
There exists $N \in \NN$, $K_{1,\min} \geq K_{\min}$, $K_{2,\min} \ge 0$ and $\varepsilon_5 > 0$, for which $\L_{\eps}^n(\B(K_1,K_2,q))\subset B(K_1,K_2,q)$ and $\L_{\eps}^n|_{\B(K_1,K_2,q)}$ is a strict contraction, with respect to the $\|\cdot\|_{1,q+1}$ norm, for all $K_i \geq K_{i, \min}$,  $|\eps| < \eps_5$, $q \in \{1,\dots,r-3\}$ and $n\geq N$.
\end{lemma}
\begin{proof} 
Let $h\in \B(K_1,K_2,q)$, then by  Proposition \ref{lem:invariance} we have $\L^n_\eps(h)\subset \B(K_1,q+1)$. Next, by Corollary \ref{cor:SLYconsist}, we have
\begin{equation}\label{eq:contractw}
\|\L_{\eps}^n (h)\|_{2,q} \leq A_1\theta^n K_2 + A_2 K_1.
\end{equation}
Choosing $N$ large enough, such that $A_1\theta^N=\beta\in(0,1)$, and $K_2$ such that $K_2 \geq K_{2, \min}:=\frac{A_2K}{1-\beta}$, we have the wanted invariance. 

Next, let $h_1,h_2\in \B(K_1,K_2,q)$. Then for $n\geq N$ we have $\L_\eps^n (h_1)=\L_{T^\eps_{h^{n-1}_1}}\dots\L_{T^\eps_{h_1}}h_1$ and $\L_\eps^n (h_2)=\L_{T^\eps_{h^{n-1}_2}}\dots\L_{T^\eps_{h_2}}h_2$ for some $h_1,\dots,h_1^{n-1}$ and $h_2,\dots,h_2^{n-1}$. Choose $\gamma \in (\theta,1)$ and $C > \max\{1,C_1\}$. We prove by induction that $\|\L_\eps^n (h_1)-\L_\eps^n (h_2)\|_{1,q+1} \leq C\gamma^n \|h_1-h_2\|_{1,q+1}$.
\begin{align*}
&\|\L_\eps^n  (h_1)-\L_\eps^n (h_2)\|_{1,q+1}=\|\L_{T^\eps_{h^{n-1}_1}}\dots\L_{T^\eps_{h_1}} h_1-\L_{T^\eps_{h^{n-1}_2}}\dots\L_{T^\eps_{h_2}} h_2\|_{1,q+1} \\
&\leq \|(\L_{T^\eps_{h^{n-1}_1}}\dots\L_{T^\eps_{h_1}}-\L_{T^\eps_{h^{n-1}_2}}\dots\L_{T^\eps_{h_2}})h_1\|_{1,q+1}\\
&\hskip 2cm+\|\L_{T^\eps_{h^{n-1}_2}}\dots\L_{T^\eps_{h_2}}(h_1-h_2)\|_{1,q+1} \\
&\leq K_2|\eps|\sum_{k=1}^n \theta^{n-k}\gamma^k \|h_1-h_2\|_{1,q+1}+C_1\theta^{n}\|h_1-h_2\|_{1,q+1},
\end{align*}
where for the first term we used \eqref{eq:StrongStatStabIter} from Corollary \ref{cor:StrongStatStab}, $\|h_1\|_{2,q} \leq K_2$ and the induction hypothesis, while for the second term we used Lemma \ref{lem:Contraction}.

Accordingly,
\begin{equation}\label{eq:ufff}
\begin{split}
\|\L_\eps^n  (h_1)-\L_\eps^n (h_2)\|_{1,q+1}\leq &K_2|\eps|\gamma^n \sum_{k=1}^n (\theta/\gamma)^k \|h_1-h_2\|_{1,q+1}+C_1\theta^{n}
\|h_1-h_2\|_{1,q+1} \\
&\leq (K_2|\eps|(1-\theta/\gamma)^{-1}+C_1)\gamma^{n}\|h_1-h_2\|_{1,q+1} \\
&\leq C\gamma^{n}\|h_1-h_2\|_{1,q+1}
\end{split}
\end{equation}
for $|\eps|$ sufficiently small.

\end{proof}

\begin{proof}[\bf Proof of Theorem \ref{thm:unifix}]
We are now in a position to show that $\L_\eps$ has a unique fixed point in $\B(K_1,K_2,q)$, for $K_i\geq K_{i,\min}$. By Lemma \ref{lem:Sinvariance} and the Banach fixed point theorem it follows that, for $N$ large enough, $\L_\eps^N$ has a unique fixed point in $\B(K_1,K_2,q)$, call it $h_\eps$. Then
\[
\L_\eps^{N}(\L_\eps (h_\eps))= \L_\eps(\L_\eps^{N} (h_\eps))=\L_\eps (h_\eps).
\]
Accordingly, also $\L_\eps h_\eps$ is a fixed point of $\L_\eps^{N}$. On the other hand, by equations \eqref{eq:contractw0},\eqref{eq:contractw} there exist constants $A,B,A_1,A_2$ such that $\L_\ve (h_\ve)\in\B(K_1',K_2',q)\supset \B(K_1,K_2,q)$, where $K_1'=AK_1+B$ and $K_2'=A_1 K_2+A_2 K_1$.  If $N$ has been chosen large enough, $\L_\eps^N$ must have a unique fixed point in $\B(K_1',K_2',q)$ as well, which must be $h_\eps$. It follows $\L_\eps (h_\eps)=h_\eps$. On the other hand, if $g_\eps\in \B(K_1,K_2,q)$ and $\L_\eps (g_\eps)=g_\eps$, then $\L_\eps^N (g_\eps)=g_\eps$ and so, by unicity again, $g_\eps=h_\eps$.


The proof of the first part of Theorem \ref{thm:unifix} is completed by \eqref{eq:ufff} that implies
\begin{equation}
\|\L_\eps^n(h) - h_\eps\|_{1,q+1} \leq C\gamma^n \|h-h_\eps\|_{1,q+1}
\end{equation}
for all $h \in \bar\B(K_1,K_2,q)$.

We now prove the second part of Theorem \ref{thm:unifix}. Let $h \in \bar\B(K_1,K_2,q)$. Then
\begin{align*}
\|\L_\eps^n(h) - h_{\eps'}\|_{1,q+1} &\leq \|\L_\eps^n(h) - \L_{\eps'}^n(h)\|_{1,q+1} + \|\L_{\eps'}^n(h) - h_{\eps'}\|_{1,q+1} \\
&\leq \|\L_\eps^n(h) - \L_{\eps'}^n(h)\|_{1,q+1} + C\theta^n\|h - h_{\eps'}\|_{1,q+1}.
\end{align*}
Choose $\lambda \in (0,1)$ and fix $n^*$ such that $C\theta^{n^*} < \lambda$. Then
\[
\|\L_\eps^{n^*}(h) - h_{\eps'}\|_{1,q+1} \leq \|\L_\eps^{n^*}(h) - \L_{\eps'}^{n^*}(h)\|_{1,q+1} + \lambda \|h - h_{\eps'}\|_{1,q+1}.
\]
According to Corollary \ref{cor:StrongStatStab}, for each  $n \in \mathbb{N}$ there exists $C(n) > 0$ such that if $h \in \bar\B(K_1,K_2,q)$,
\[
\|\L_\eps^n(h) - \L_{\eps'}^n(h)\|_{1,q+1} \leq C(n) |\eps-\eps'|.
\]
We thus have
\[
\|\L_\eps^{n^*}(h) - h_{\eps'}\|_{1,q+1} \leq C(n^*)|\eps-\eps'| + \lambda \|h - h_{\eps'}\|_{1,q+1}.
\]
Let $B(h,r) = \{g \in \B^{1,q}: \|g-h\|_{1,q} \leq r \}$. Set $r_0 = \frac{C(n^*)|\eps-\eps'|}{1-\lambda}$. Then $\L^{n^*}_{\eps}B(h_{\eps'},r_0) \subseteq B(h_{\eps'},r_0)$. Indeed, let $g$ be such that $\|g-h_{\eps'}\|_{1,q+1} \leq r_0$. Then
\[
\|\L^{n^*}_{\eps}(g)-h_{\eps'}\|_{1,q+1} \leq C(n^*)|\eps-\eps'| +\lambda r_0 = r_0.
\]
According to the main theorem, we have $\|\L^{kn^*}_{\eps}(h)-h_{\eps}\|_{1,q} \to 0$ for $h \in \overline{\B}(K_1,K_2,q)$. Note that $B(h_{\eps'},r_0) \cap \bar\B(K_1,K_2,q) \neq \emptyset$. This implies that $h_{\eps} \in B(h_{\eps'},r_0)$, i.e.
\[
\|h_{\eps}-h_{\eps'}\|_{1,q+1} \leq r_0 = \frac{C(n^*)|\eps-\eps'|}{1-\lambda} := K|\eps-\eps'|.
\]
\end{proof}

\subsection{Uniqueness of the physical measure}\ \\
We can now conclude with the proof of Theorem \ref{thm:main0}.
\begin{proof}[Proof of Theorem \ref{thm:main0}]
Let $h_\eps$ be a physical measure and let $h\in L^1$ such that $\L_\ve^n (h)$ converges weakly to $h_\eps$.
For each $\delta>0$ we can find $h_\delta\in \cC^\infty$ such that $\|h_\delta-h\|_{L^1}\leq \delta$. Then, setting $h_n=\L_\ve^n (h)$, we have
\[
\|\L_{T_{h_n}^\eps}\cdots \L_{T_h^\eps} h_\delta-\L_{T_{h_n}^\eps}\cdots \L_{T_h^\eps} h\|_{L^1}\leq C \delta.
\]
It follows that, given any sequence $\delta_n\to 0$, $\L_{T_{h_n}^\eps}\cdots \L_{T_h^\eps} h_{\delta_n}$ converges weakly to $h_\eps$.
Moreover,
\[
\|\L_{T_{h_n}^\eps}\cdots \L_{T_h^\eps}h_\delta\|_{1,q}\leq C\theta^n\|h_\delta\|_{1,q}+B\|h_\delta\|_{0,q}\leq
C_\delta\theta^n+B\|h_\delta\|_{TV}\leq C_\delta\theta^n+B,
\]
where $C_\delta\to\infty$ when $\delta\to 0$ and depends only on the fixed function $h$. We can thus choose a sequence $\delta_n$ that goes to zero so slowly that $C_{\delta_n}\theta^n\to 0$ when $n\to\infty$.
Thus, for $n$ large enough, $\|\L_{T_{h_n}^\eps}\cdots \L_{T_h^\eps} h_{\delta_n}\|_{1,q}\leq 2B\leq K_1$.
On the other hand, computing as above
\[
\|\L_{T_{h_n}^\eps}\cdots \L_{T_h^\eps} h_{\delta_n}\|_{2,q}\leq  K_2.
\]
The above implies that there exists a sequence $\{g_n\}\subset \B(K_1,K_2,q)\cap \C^r$ such that $g_n$ converges weakly to $h_\eps$.

Next, notice that $\L_{T_{h_\eps}^\eps}$, the transfer operator associated with the Anosov map $T_{h_\eps}^\eps$. Since $T_{h_\eps}^\eps$ is a small perturbation of $T$, which is transitive, by structural stability $T_{h_\eps}^\eps$ is transitive as well. It follows that $\L_{T_{h_\eps}^\eps}$, when acting on $\B^{1,q}$, has a spectral gap, see \cite{BL21}, and consequently $T_{h_\eps}^\eps$ has a unique invariant measure in $\B^{1,q}$, call it $g_*$.
Therefore there exists constants $C_*>0$ and $\nu\in(0,1)$ such that, for each $n\in\bN$,  we have
\[
\|\L_{T_{h_\eps}^\eps}^{m}g_n-g_*\|_{1,q+1}\leq C_*\nu^m.
\]
This implies that, for each $(W,\varphi)\in\Omega_{L,q,1}$, we have (recalling that $h_\eps=\L_\ve (h_\eps)=\L_{T_{h_\eps}^\eps}h_\eps$)
\[
\int (g_*-h_\eps)\vf=\int (\L_{T_{h_\eps}^\eps}^{m}g_n-\L_{T_{h_\eps}^\eps}^{m}h_\eps)\vf+\int ( \L_{T_{h_\eps}^\eps}^{m}g_n-g_*)\vf.
\]
From the above, it follows
\[
\left| \int (g_*-h_\eps)\vf-\int (g_n-h_\eps)\vf\circ (T_{h_\eps}^\eps)^{m}\right|\leq C_*\nu^m.
\]
Taking the limit for $n\to\infty$, since $g_n$ converges weakly to $h_\eps$ we have 
\[
\left| \int (g_*-h_\eps)\vf \right|\leq C_*\nu^m,
\]
and taking the limit $m\to \infty$ we have $g_*=h_\eps$ in $\B^{0,q}$, hence, recalling \cite[Lemma 2.12]{BL21}, they are equal as distributions in $(\C^q)'$. But since $g_*$ is a positive distribution they are both measures and hence they coincide as measures.  
It follows that $h_\eps=g_*\in \overline{\B}(K_1,K_2,q)$ and since the invariant measure in such a set is unique the theorem follows.
\end{proof}

\appendix

\section{Test functions and foliations}
\subsection{Test functions.}\ \\
 Let  $\param\geq 2$ be a parameter chosen as in \cite[Equation (3.11)]{BL21}. Denote by $\alpha$ the multi-index  $\alpha=(\alpha_1,\cdots, \alpha_d)$ with $\alpha_i\in \mathbb N\cup\{0\}$. Let $|\alpha|=\sum_{i=1}^d\alpha_i$ and $\partial^{\alpha}=\partial_{x_1}^{\alpha_1}\cdots\partial_{x_d}^{\alpha_d}$. We thus define the weighted norm in $\mathcal C^\rho(M,\mathcal M(m,n))$, where $\mathcal M(m,n)$ is the set of the $m\times n$ (possibly complex valued) matrices,  
\begin{equation}
\label{def of the C norm}
\begin{split}
&\|\varphi \|_{\mathcal{C}^0}=\sup_{x \in M}\sup_{i\in\{1, \dots,n\}}\sum_{j=1}^m|\varphi_{i,j}(x)| \\
&\|\varphi\|_{\mathcal{C}^{\rho}}=\sum_{k=0}^{\rho}\param^{\rho-k} \sup_{|\alpha|=k}\|\partial^{\alpha}\varphi\|_{\mathcal{C}^0},
\end{split}
\end{equation}
for some $\param\ge 2$. Note that the above definition implies
\begin{equation}\label{eq:norm-q-def}
\|\varphi\|_{\mathcal{C}^{\rho+1}}=\param^{\rho+1} \|\varphi\|_{\mathcal{C}^{0}}+\sup_{i}\|\partial_{x_i} \varphi\|_{\mathcal{C}^{\rho}}.
\end{equation}
 The next lemma is Lemma 2.9 of \cite{BL21}.
\begin{lemma}\label{lem:prop-norm}
For every $ \rho,n, m,s  \in\NN$, $ \psi\in\mathcal C^\rho(M,\mathcal M(m,n))$ and $\varphi \in \mathcal C^\rho(M,\mathcal M(m,s))$ we have 
\[
\|\varphi\psi\|_{\mathcal{C}^{\rho}}\le \|\varphi\|_{\mathcal{C}^{\rho}}\|\psi\|_{\mathcal{C}^{\rho}}.
\]
Moreover if $\varphi\in \mathcal C^\rho(M,\mathcal M(m,n))$ and $\psi\in\mathcal C^\rho(M,M)$, then
\[
\|\varphi\circ \psi\|_{\mathcal C^\rho}\le \sum_{k=0}^{\rho}\binom{\rho}{k} \param^{\rho-k}\|\varphi\|_{\mathcal C^{k}}\prod_{i=1}^{k}\|(D\psi)^t\|_{\mathcal C^{\rho-i}}. 
\]
\end{lemma}

\begin{lemma} \label{lem:lip1}
		Let $\varphi \in \C^{k}(M,\CC)$ and $\psi,\tilde \psi \in \C^k(M,M)$. Then
\begin{align}
&\|\varphi \circ \psi - \tilde \varphi \circ \psi \|_{\C^k} \leq \sum_{j=0}^k \omega^j \sup_{|\alpha|=k-j}\|\partial^{\alpha}\varphi - \partial^{\alpha}\tilde \varphi\|_{\C^0} \prod_{i=j}^{k-1}\|(D\psi)^t\|_{\C^i} \label{eq:lip1}
\end{align}
and
\begin{align}
&\|\varphi \circ \psi - \varphi \circ \tilde \psi \|_{\C^k} \nonumber \\
&\leq \sum_{j=0}^k \omega^j \sup_{|\alpha|=k-j}\|(\partial^{\alpha}\varphi) \circ \psi - (\partial^{\alpha}\varphi) \circ \tilde \psi\|_{\C^0} \prod_{i=j}^{k-1}\|(D\psi)^t\|_{\C^i} \nonumber \\
&+\sum_{j=1}^k \sup_{|\alpha|=k-j} \|\partial^{\alpha}\varphi \circ \tilde \psi\|_{\C^{j-1}} \|(D\psi)^t-(D\tilde \psi)^t\|_{\C^{j-1}}\prod_{i=j}^{k-1}\|(D\psi)^t\|_{\C^{i}}. \label{eq:lip2}
\end{align}
\end{lemma}
\begin{proof}
	We are going to prove both formulas by induction. First recall that
	\begin{equation}
	\|\eta\|_{\C^{\rho+1}}=\param^{\rho+1}\|\eta\|_{\C^0}+\sup_i\|\partial_{x_i}\eta\|_{\C^{\rho}}. \label{eq:ind}
	\end{equation}
	We first prove \eqref{eq:lip1}. Using the above formula we compute
	\begin{align*}
	&\|\varphi \circ \psi - \tilde \varphi \circ \psi \|_{\C^{k+1}}\\
	&\leq\param^{k+1}\|\varphi - \tilde \varphi \|_{\C^0}+\sup_i\|\partial_{x_i}(\varphi \circ \psi - \tilde \varphi \circ  \psi )\|_{\C^{k}} \\
	&\leq \param^{k+1}\|\varphi - \tilde \varphi \|_{\C^0}+\sup_i\|\partial_{x_i}\varphi - \partial_{x_i}\tilde \varphi\|_{\C^{k}}\|(D\psi)^t\|_{\C^k} \\
	& \leq \param^{k+1}\|\varphi - \tilde \varphi \|_{\C^0}\\
	&+\sup_i\sum_{j=0}^k \omega^j \sup_{|\alpha|=k-j}\|\partial^{\alpha}\partial_{x_i}\varphi - \partial^{\alpha}\partial_{x_i}\tilde\varphi\|_{\C^0} \prod_{i=j}^{k-1}\|(D\psi)^t\|_{\C^i} \cdot \|(D\psi)^t\|_{\C^k} \\
	& \leq \param^{k+1}\|\varphi - \tilde \varphi \|_{\C^0}+\sum_{j=0}^k \omega^j \sup_{|\alpha|=k+1-j}\|\partial^{\alpha}\varphi - \partial^{\alpha}\tilde\varphi\|_{\C^0} \prod_{i=j}^{k}\|(D\psi)^t\|_{\C^i}\\
	&=\sum_{j=0}^{k+1} \omega^j \sup_{|\alpha|=k+1-j}\|\partial^{\alpha}\varphi - \partial^{\alpha}\tilde\varphi\|_{\C^0} \prod_{i=j}^{k}\|(D\psi)^t\|_{\C^i}
	\end{align*}
	We now prove \eqref{eq:lip2} by using \eqref{eq:ind}.
	\begin{align*}
	\|\varphi \circ \psi - \varphi \circ \tilde \psi \|_{\C^{k+1}}&=\param^{k+1}\|\varphi \circ \psi - \varphi \circ \tilde \psi \|_{\C^0}+\sup_i\|\partial_{x_i}(\varphi \circ \psi - \varphi \circ \tilde \psi )\|_{\C^{k}} \\
	&\leq \param^{k+1}\|\varphi \circ \psi - \varphi \circ \tilde \psi \|_{\C^0}\\
	&+\sup_i\|(\partial_{x_i}\varphi) \circ \psi - (\partial_{x_i}\varphi) \circ \tilde \psi \|_{\C^{k}}\|(D\psi)^t\|_{\C^k}\\
	&+\sup_i\|(\partial_{x_i}\varphi) \circ \tilde \psi \|_{\C^{k}}\|(D\psi)^t-(D\tilde \psi)^t\|_{\C^k}
	\end{align*}	
	For the second term we use the inductive assumption. This gives us the following:
	\begin{align*}
	&\|\varphi \circ \psi - \varphi \circ \tilde \psi \|_{\C^{k+1}} \\
	&\leq \param^{k+1}\|\varphi \circ \psi - \varphi \circ \tilde \psi \|_{\C^0}\\
	&+ \sum_{j=0}^k \omega^j \sup_{|\alpha|=k-j,i}\|(\partial^{\alpha}\partial_{x_i}\varphi) \circ \psi - (\partial^{\alpha}\partial_{x_i}\varphi) \circ \tilde \psi\|_{\C^0} \prod_{i=j}^{k}\|(D\psi)^t\|_{\C^i} \\
	&+\sum_{j=1}^k \sup_{|\alpha|=k-j,i} \|\partial^{\alpha}\partial_{x_i}\varphi \circ \tilde \psi\|_{\C^{j-1}} \|(D\psi)^t-(D\tilde \psi)^t\|_{\C^{j-1}}\prod_{i=j}^{k}\|(D\psi)^t\|_{\C^{i}} \\
	&+\sup_i\|(\partial_{x_i}\varphi) \circ \tilde \psi \|_{\C^{k}}\|(D\psi)^t-(D\tilde \psi)^t\|_{\C^k}\\
	&= \param^{\rho+1}\|\varphi \circ \psi - \varphi \circ \tilde \psi \|_{\C^0}\\
	&+ \sum_{j=0}^k \omega^j \sup_{|\alpha|=k+1-j}\|(\partial^{\alpha}\varphi) \circ \psi - (\partial^{\alpha}\varphi) \circ \tilde \psi\|_{\C^0} \prod_{i=j}^{k}\|(D\psi)^t\|_{\C^i} \\
	&+\sum_{j=1}^k \sup_{|\alpha|=k+1-j} \|\partial^{\alpha}\varphi \circ \tilde \psi\|_{\C^{j-1}} \|(D\psi)^t-(D\tilde \psi)^t\|_{\C^{j-1}}\prod_{i=j}^{k}\|(D\psi)^t\|_{\C^{i}} \\
	&+\sup_i\|(\partial_{x_i}\varphi) \circ \tilde \psi \|_{\C^{k}}\|(D\psi)^t-(D\tilde \psi)^t\|_{\C^k} 
\end{align*}
\begin{align*}
	&= \sum_{j=0}^{k+1} \omega^j \sup_{|\alpha|=k+1-j}\|(\partial^{\alpha}\varphi) \circ \psi - (\partial^{\alpha}\varphi) \circ \tilde \psi\|_{\C^0} \prod_{i=j}^{k}\|(D\psi)^t\|_{\C^i} \\
	&+\sum_{j=1}^{k+1} \sup_{|\alpha|=k+1-j} \|\partial^{\alpha}\varphi \circ \tilde \psi\|_{\C^{j-1}} \|(D\psi)^t-(D\tilde \psi)^t\|_{\C^{j-1}}\prod_{i=j}^{k}\|(D\psi)^t\|_{\C^{i}}.
	\end{align*}
\end{proof}
We obtain a useful corollary:
\begin{corollary} \label{cor:lip}
As a consequence of \eqref{eq:lip1}, we have
\begin{equation}
\|\varphi \circ \psi - \tilde \varphi \circ \psi \|_{\C^k} \leq C\|\varphi-\tilde \varphi\|_{\C^{k}}  \|D\psi\|_{\C^{k}}^k. \label{eq:lip1cor}
\end{equation}		
If furthermore $\varphi \in \C^{k+1}(M,\CC)$, then by \eqref{eq:lip2} we have
\begin{equation}
\begin{split} 
\|\varphi \circ \psi - \varphi \circ \tilde \psi \|_{\C^k}  \leq & C\|\varphi\|_{\C^{k+1}}  \|D\psi\|_{\C^{k}}^k\\
&\times \left(\|\psi-\tilde \psi \|_{\C^0}+\|D\tilde \psi\|_{\C^{k}}^k\|D\psi-D\tilde \psi\|_{\C^{k-1}}\right). \label{eq:lip2cor}
\end{split}
\end{equation}	
\end{corollary}

\subsection{Foliations}\label{sec:foliation}\ \\ 
To define the anisotropic spaces $\B^{0,q}$ and $\B^{1,q}$, we need to define a class of (stable) foliations adapted to the cone field, whose representation in local coordinates has certain uniform regularity. Let us recall be basic defintions from \cite{BL21}.


\begin{definition} \label{def:foliation0}
A $\mathcal C^r$ $t$-dimensional foliation $W$ is a collection $\{W_\alpha\}_{\alpha\in A}$, for some set $A$, such that the $W_\alpha$ are pairwise disjoint, $\cup_{\alpha\in A}W_\alpha=M$ and for each $\xi\in W_\alpha$ there exists a neighborhood $B(\xi)$ such that the connected component of $W_\alpha \cap B(\xi)$ containing $\xi$, call it $W(\xi)$, is a $\cC^r$ $t$-dimensional open submanifold of $M$. We will call $\cF^r$ the set of $\cC^r$ $d_s$-dimensional foliations.
\end{definition}

\begin{definition}\label{def:F}
A foliation $W$ is adapted to the cone field $\cC$ if, for each $\xi\in M$, $T_\xi W(\xi)\subset C(\xi)$. Let $\cF^r_\cC$ be the set of $\cC^r$ $d_s$-dimensional foliations adapted to $\cC$.
\end{definition}

Given a $d_s$-foliation adapted to $\cC$ we can associate to it local coordinates as follows. Let $\delta_0>0$ be sufficiently small so that for each $\xi\in M$ there exists a chart $(V_i,\phi_i)$ with $\xi\in V_i$ and such that $U_i:=\phi_i(V_i)$ contains the ball $B_{\delta_0}(\phi_i(\xi))$.\footnote{Here, and in the following, we use $B_\delta(x)$ to designate $\{z\in\bR^{d'}\:\;: \|x-z\|\leq \delta\}$ for any $d'\in\bN$.} Also, choose $U^0=U^0_u\times U^0_s\subset \bR^{d_u}\times\bR^{d_s}$ with $U^0_u= B_{\delta_0/2}(0)$, $U^0_s= B_{\delta_0/2}(0)$. Next, for each $z\in U_i$, let $W(z)$ be the connected component of $\phi_i(W)$ containing $z$.\footnote{ Refer to Definition \ref{def:foliation0} for the exact meaning of ``connected component". Also note the abuse of notation since we use the same name for the sub-manifond in $M$ and its image in the chart.} Define the function $F_\xi:U^0\to\bR^{d_u}$ by $\{ (F_\xi(x,y)+x_\xi,y_\xi+y)\}=\{(w,y+y_\xi)\}_{w\in \bR^{d_u}}\cap W(x+x_\xi,y_\xi)$, where $(x_\xi,y_\xi)=\phi_i(\xi)$.\footnote{ The fact that the intersection is non void and consists of exactly  one point follows trivially from the fact that the foliation is adapted to the cone field, hence the two manifolds are transversal.}
That is, $W(x+x_\xi,y_\xi)$ is exactly  the graph of the function $F_\xi(x,\cdot)+x_\xi$. Moreover, 
 \begin{equation}\label{eq:F-def}
 F_\xi(x,0)=x.
 \end{equation}
In addition, we ask $\delta_0$ to be small enough that the expression of $DT$ in the above charts is roughly constant. See Lemma B.5 \cite{BL21}.


Now $\mathbb F_\xi(x,y)=(F_\xi(x,y),y)$, $(x,y) \in U^0$ describes the foliation locally. Denote by $\mathbb F$ the collection of maps $\{\mathbb F_\xi\}$.

For each integer $r\geq 2$ and $L>0$, let
\[
\begin{split}
&\overline{\mathcal F}^r_\C=\left\{W\in\mathcal F^r_\C:  \mathbb F \in\C^r(U^0,\RR^{d})\right\}\\
&\mathcal W_L^r=\Big\{W\in\overline{\mathcal F}^r_\C :  \sup_\xi \sup_{x\in U^0_u} \sup_{|\alpha|=k}\|\partial^\alpha_y F_\xi(x,\cdot)\|_{\C^0(U^0_s,\RR^{d_u})}\leq L^{(k-1)^2} , 2\leq k\leq r;\\
 &\hskip 1.5cm \sup_\xi \sup_{x\in U^0_u} \sup_{|\alpha|=k}\|\partial^\alpha_y  H^{F_\xi}(x,\cdot)\|_{\C^{0}(U^0_s,\RR^{d_s})}\leq L^{(k+1)^2}, 0\leq k\leq r-2\Big\}, 
\end{split}
\]
where 
\begin{equation*}
\begin{split}
 H^{F_{\xi}}(x,y)&=\sum_{j=1}^{d_u}\left[\partial_{x_j}\left(\left[\partial_y (F_{\xi})_j\right]\circ \mathbb F_\xi^{-1}\right)\right]\circ \mathbb F_{\xi}(x,y)\\
 &=\sum_{ij}\partial_{x_i}\partial_y (F_{\xi})_j\cdot (\partial_x F_{\xi})^{-1}_{ij}.
\end{split}
\end{equation*}
For each $\varphi\in\C^{r}(M, \CC^l)$ and $W\in\mathcal F^r_{\C}$ let $\varphi_{\xi,x}(\cdot)=\varphi\circ \phi_{i}^{-1}\circ \FF_{\xi}(x,\cdot)$, $q\le r$ and define
\begin{equation}\label{eq:norm_dual}
\| \varphi\|_q^W:=\sup_{\xi\in M}\sup_{x\in U^0_u}\|\varphi_{\xi,x}\|_{\C^q(U^0_s, \CC^{l})}=\sup_{\xi\in M}\sup_{x\in U^0_u}\sum_{j=1}^l\|(\varphi_{\xi,x})_j\|_{\C^q(U^0_s, \CC)}.
\end{equation}
We are finally able to define the sets $\Omega_{L,q,l}$, $L > 0$, $r \geq 2$, $q \in \NN \cup \{0\}$ as
\begin{equation}\label{eq:measure}
\Omega_{L,q,l}=\left\{(W,\varphi)\in \mathcal W_L^r\times \C^q(M, \CC^l)\;:\;\|\varphi\|_q^W\le 1\right\}.
\end{equation}
\section{Some properties of the coupled map $T_h^{\eps}$}
\begin{lemma}\label{lem:anosovcoup}
There exists $\eps^* > 0$ such that 
\begin{equation} \label{eq:seconeinv}
D_{\xi}(T_h^\eps)^{-1} C(\xi)\subset \text{int}(C((T_h^\eps)^{-1}(\xi)))\cup\{0\}
\end{equation}
for all $|\eps|<\eps^*$ and $h\in \B_1^{0,q}$;
moreover there exists $\lambda>1, \nu\in (0,1)$, $c\in(0, 1)$ such that
\begin{equation}\label{eq:anosovcoup}
\begin{split}
&\inf_{\xi\in M}\inf_{v\in C(\xi)} \|D_\xi (T_{h_{n-1}}^{\eps}\circ\cdots\circ  T_{h_0}^{\eps})^{-1}v\|> c\nu^{-n}\|v\|\\
& \inf_{\xi\in M}\inf_{v\not\in C(\xi)} \|D_\xi (T_{h_{n-1}}^{\eps}\circ\cdots\circ T_{h_0}^{\eps})v\|> c\lambda^{n}\|v\|
\end{split}
\end{equation}
for all $|\eps|<\eps^*$ and any sequence $h_0,\dots,h_{n-1}\in \B_1^{0,q}$, $n \in \mathbb{N}$.
\end{lemma}
\begin{proof}
Fix $\eps^* > 0$ such that \eqref{eq:seconeinv} holds. Define $\nu = (1-\eps^*)^{-1}\nu_0$ and $\lambda = (1-\eps^*)\lambda_0$, where $\nu_0$ and $\lambda_0$ are given by \eqref{eq:anosov}. Decrease $\eps^*$ further if necessary so that that $0<\nu<1<\lambda$.
Via a standard change of metric we may assume that in \eqref{eq:anosov} $c_0 = 1$.
Note that for any $h$, $|\eps|<\eps^*$ and $v\in C(\xi)$
\begin{align*}
\|D_\xi (T_{h}^{\eps})^{-1}v\|&\ge\|D_\xi T^{-1}v\| -\|(D_\xi T^{-1}-D_\xi (T_{h}^{\eps})^{-1})v\|\\
&> \nu_0^{-1}\|v\|-|\eps| \nu_0^{-1}\|v\|=
\nu^{-1}\|v\|.
\end{align*}
We then proceed by induction on $n$. Assume that for any $v\in C(\xi)$
$$\|D_\xi (T_{h_{n-2}}^{\eps}\circ\cdots\circ T_{h_0}^{\eps})^{-1}v\|> \nu^{-(n-1)}\|v\|.$$
Write $T_{h_{n-2}}^{\eps}\circ\cdots\circ T_{h_0}^{\eps}=T^{\eps}_{n-2}$. Then using the above two inequalities we obtain
\begin{align*}
&\|D_\xi (T_{h_{n-1}}^{\eps}\circ\cdots\circ T_{h_0}^{\eps})^{-1}v\|=\|D_\xi (T_{h_{n-1}}^{\eps}\circ T^{\eps}_{n-2})^{-1}v\| \\
&=\|D_{(T_{h_{n-1}}^{\eps})^{-1}(\xi)}(T_{n-2}^{\eps})^{-1} D_{\xi}(T_{h_{n-1}}^{\eps})^{-1} v\| \\
& > \nu^{-(n-1)}\|D_{\xi}(T_{h_{n-1}}^{\eps})^{-1} v\| > \nu^{-n}\|v\|.
\end{align*}
Similarly, for $v\notin C(\xi)$ we obtain
$$\|D_\xi (T_{h_{n-1}}^{\eps}\circ\cdots\circ  T_{h_0}^{\eps})v\|> \lambda^n\|v\|.$$
Finally, returning to the original metric accounts for the constant $c$ in the statement of the lemma.
\end{proof}
\begin{remark}\label{re:allanosov}
Lemma \ref{lem:anosovcoup} implies that for $|\eps|$ small enough, each $T_h^\eps$ is an Anosov diffeomorphism and any concatenation $T_{h_{n-1}}^{\eps}\circ\cdots \circ T_{h_0}^{\eps}$ satisfies \eqref{eq:anosov} with uniform constant $c$ independent of $\eps$.
\end{remark}
\begin{lemma}\label{lem:interanosovcoup}
Let $T_t=tT_{h_1}^\eps+(1-t)T_{h_2}^\eps$ for $t \in [0,1]$ (understood in the charts defined in the proof of Lemma \ref{lem:lip}). There exists $\eps^{**} > 0$ such that 
\begin{equation} \label{eq:coneinv}
D_{\xi}T_t^{-1} C(\xi)\subset \text{int}(C(T_t^{-1}(\xi)))\cup\{0\}
\end{equation}
for all $|\eps|<\eps^{**}$ and $h_1,h_2\in \B_1^{0,q}$;
moreover there exists $\tilde\lambda>1, \tilde\nu\in (0,1)$, $\tilde c\in(0, 1)$ such that
\begin{equation}
\begin{split}
&\inf_{\xi\in M}\inf_{v\in C(\xi)} \|D_\xi T_t^{-n}v\|> \tilde  c\tilde \nu^{-n}\|v\|\\
& \inf_{\xi\in M}\inf_{v\not\in C(\xi)} \|D_\xi T_t^n v\|> \tilde c \tilde\lambda^{n}\|v\|
\end{split}
\end{equation}
for all $|\eps|<\eps^{**}$, $n \in \mathbb{N}$.
\end{lemma}
\begin{proof}
To simplify notation, let $T_i=T_{h_i}^\eps$. Write $T_2^{-1} \circ T_1 = Id + \eps S$. Then $T_t = [Id + t \cdot \eps S] \circ T_1$ and $T^{-1}_t = T_1^{-1} \circ [Id + t \cdot \eps S]^{-1}$. 

We can see now that it is possible to fix $\eps^{**} > 0$ such that \eqref{eq:coneinv} holds. Define $\tilde \nu = (1-\eps^{**})^{-1}\nu$ and $\tilde \lambda = (1-\eps^{**})\lambda$, where $\nu$ and $\lambda$ are given by \eqref{eq:anosovcoup}. Decrease $\eps^{**}$ further if necessary so that that $0<\tilde \nu<1<\tilde \lambda$.
Via a standard change of metric we may assume that in \eqref{eq:anosov} $c = 1$.
Note that for any $h_1,h_2$, $|\eps|<\eps^{**}$ and $v\in C(\xi)$
\begin{align*}
\|D_\xi T_t^{-1}v\|&\ge\|D_\xi T_1^{-1}v\| -\|(D_\xi T_1^{-1}-D_\xi ( T_1^{-1} \circ [Id + t \cdot \eps S]^{-1}))v\|\\
&> \nu^{-1}\|v\|-|\eps| \nu^{-1}\|v\|=
\tilde \nu^{-1}\|v\|.
\end{align*}
Similarly, for $v\notin C(\xi)$ we obtain
$$\|D_\xi T_t^n v\|> \tilde \lambda^n\|v\|.$$
Finally, returning to the original metric accounts for the constant $c$ in the statement of the lemma.
\end{proof}
\section{Projection along the unstable direction} \label{sec:projection}
Here we follow \cite{BL21} and introduce a way to project along the approximate stable and unstable directions.
We do this by introducing projectors $\pi^u,\pi^s$ which are only implicit in \cite{BL21}.
Note that the construction is local, so we can argue in one chart without further mentioning it. We start by recalling the construction in \cite{BL21}.

Consider the ``almost unstable" foliation $\Gamma=\{\gamma_s\}_{s\in\RR^{d_s}}$ made of the leaves $\gamma_s=\{(u,s)\}_{u\in\RR^{d_u}}$ and its image $T^n \Gamma$. The leaves of $T^n \Gamma$ can be expressed in the form $\{(x, \tilde G_n(x,y)\}$ for some function $\tilde G_n$, smooth in the $x$ variable,  with $\|\partial _x \tilde G_n\|\leq 1$ and the normalization $\tilde G_n(F(0,y),y)=y$. On the other hand, the leaves of $W$ have the form $\{(F(x,y),y)\}$. It is then natural to consider the change of variables $(x,y)=\Psi_n(x',y')$ where $(x,\tilde G_n(x,y'))=(F(x',y), y)$. Writing $\vf=(\vf_1,\vf_2)$, with $\vf_1\in \RR^{d_u}$, $\vf_2\in\RR^{d_s}$ we consider the decomposition defined in \cite[Equations (3.5), (3.6)]{BL21},\footnote{Note that in \cite{BL21} the projectors where not explicitly defined.}
\begin{equation}\label{eq:vfvw}
\begin{split}
\vf&(x,y)=:\pi^u\vf(x,y)+\pi^s\vf(x,y)\\
&=(v(x,y), \partial_x \tilde G_n(x,y') v(x,y))+(\partial_y F(x',y) w(x,y), w(x,y)).
\end{split}
\end{equation}
Where, 
\[
\begin{split}
&v(x,y)=(\Id-\partial_y F(x',y) \partial_x\tilde G_n(x,y'))^{-1}(\vf_1(x,y)-\partial_y F(x',y) \vf_2(x,y))\\
&w(x,y)= (\Id-\partial_x \tilde G_n(x,y')\partial_y F(x',y))^{-1}(\vf_2(x,y)-\partial_x \tilde G_n(x,y') \vf_1(x,y)).
\end{split}
\]
Let us check that $\pi^u,\pi^s$ are indeed projectors. Note that
\[
\begin{split}
(\pi^u \vf)_1-\partial_y F (\pi^u \vf)_2&=[\Id-\partial_y F \partial_x\tilde G_n](\Id-\partial_y F \partial_x\tilde G_n)^{-1}(\vf_1-\partial_y F \vf_2)\\
&=\vf_1-\partial_y F \vf_2,
\end{split}
\]
which immediately implies $(\pi^u)^2=\pi^u$. The computation for $\pi^s$ is similar.

The key properties of the above projectors are as follows.

By \cite[Equation (3.10)]{BL21} we have, for $(W,\vf)\in\Omega_{L,q+2,d}$,
\begin{equation}\label{eq:piu}
\|(DT^{-n})^{-1}\circ T^{-n}\pi^u\vf\|^{W}_{q+2}\leq C\lambda^{-n}\|\vf\|^W_{q+2}+\frac{C_n}{\varpi}\| \vf\|^{W}_{q+1}.
\end{equation}
In addition, by \cite[Equation (3.8)]{BL21}, we have
\begin{equation}\label{eq:pis}
\|\sum_{i=1}^d\partial_{x_i}[(DT^{-n})^{-1}\pi^s\circ T^{n}\vf\circ T^{n}]_i\|^{T^{-n}W}_{q+1}\leq C_n\|\vf\|^W_{q+2}.
\end{equation}
The second of \cite[Equation (3.8)]{BL21} implies also
\begin{equation}\label{eq:pis2}
\|(DT^{-n})^{-1}\circ T^{-n}\pi^s \vf\|^{W}_{q+2}\leq C_{n,\varpi}\|\vf\|^W_{q+2}.
\end{equation}


\begin{thebibliography}{99}
\bibitem{BL21} Bahsoun, W., \& Liverani, C. (2022). Anosov diffeomorphisms, anisotropic BV spaces and regularity of foliations. \emph{Ergodic Theory and Dynam. Systems}, 42(8), 2431-2467.
\bibitem{Ba} Baladi, V. (2018). \emph{Dynamical zeta functions and dynamical determinants for hyperbolic maps}. Springer International Publishing.
\bibitem{BT} Baladi, V., \& Tsujii, M. (2007). Anisotropic H\"older and Sobolev spaces for hyperbolic diffeomorphisms. In \emph{Annales de l'institut Fourier} (Vol. 57, No. 1, pp. 127-154).
\bibitem{BKST18} B\'alint, P., Keller, G., S\'elley, F. M., \& T\'oth, I. P. (2018). Synchronization versus stability of the invariant distribution for a class of globally coupled maps. \emph{Nonlinearity}, 31(8), 3770--3793.
\bibitem{BKZ} Bardet, J. B., Keller, G., \& Zweim\"uller, R. (2009). Stochastically stable globally coupled maps with bistable thermodynamic limit. \emph{Communications in Mathematical Physics}, 292(1), 237-270.
\bibitem{BA11} Bick, C., Timme, M., Paulikat, D., Rathlev, D., \& Ashwin, P. (2011). Chaos in symmetric phase oscillator networks. \emph{Physical Review Letters}, 107(24), 244101.
\bibitem{Bi21} Bick, C., Gross, E., Harrington, H. A., \& Schaub, M. T. (2021). What are higher-order networks?. arXiv preprint arXiv:2104.11329.
\bibitem{Blank} Blank, M. L. (2011, February). Self-consistent mappings and systems of interacting particles. In \emph{Doklady Mathematics} (Vol. 83, No. 1, pp. 49-52). SP MAIK Nauka/Interperiodica.
\bibitem{BKL} Blank, M., Keller, G., \& Liverani, C. (2002). Ruelle-Perron-Frobenius spectrum for Anosov maps. \emph{Nonlinearity}, 15(6), 1905.
\bibitem{BS88} Bunimovich, L. A., \& Sinai, Y. G. (1988). Spacetime chaos in coupled map lattices. \emph{Nonlinearity}, 1(4), 491.
\bibitem{DKL21} Demers, M. F., Kiamari, N., \& Liverani, C. (2021). Transfer operators in hyperbolic dynamics. An introduction., 33 Colloquio Brasilero de Matematica. Brazilian Mathematics Colloquiums series, Editora do IMPA. pp.252. ISBN 978-65-89124-26-9.
\bibitem{DZ15} Dyatlov, S., \& Zworski, M. (2015). Stochastic stability of Pollicott-Ruelle resonances. \emph{Nonlinearity}, 28(10), 3511.
\bibitem{Euler} Euler, L. (1757). Principes g\'en\'eraux du mouvement des fluides. \emph{M\'emoires de l'acad\'emie  des sciences de Berlin}, 274-315.
\bibitem{B14} Fernandez, B. (2014). Breaking of ergodicity in expanding systems of globally coupled piecewise affine circle maps. \emph{Journal of Statistical Physics}, 154(4), 999-1029.
\bibitem{G21} Galatolo, S. (2022). Self-consistent transfer operators: Invariant measures, convergence to equilibrium, linear response and control of the statistical properties. Communications in Mathematical Physics, 395(2), 715-772.
\bibitem{GaL} Galatolo, S., \& Lucena, R. (2020). Spectral Gap and quantitative statistical stability for systems with contracting fibers and Lorenz like maps. \emph{Discrete Continuous Dynamical Systems}, 40, no. 3, 1309--1360. 
\bibitem{Ga} G\"artner, J. (1988). On the McKean-Vlasov limit for interacting diffusions. \emph{Mathematische Nachrichten}, 137(1), 197-248.
\bibitem{G16} Golse, F. On the dynamics of large particle systems in the mean-field limit.
In Macroscopic and large scale phenomena: coarse graining, mean field limits and ergodicity.
Lecture Notes in Applied Mathematics and Mechanics (A. Muntean, J.
Rademacher and A. Zagaris Eds.) vol. 3, p. 1--144, Springer, 2016.
\bibitem{GL} Gou\"ezel, S., \& Liverani, C. (2006). Banach spaces adapted to Anosov systems. \emph{Ergodic Theory and Dynamical Systems}, 26(1), 189-217.
\bibitem{GL1} Gou\"ezel, S., \& Liverani, C. (2008). Compact locally maximal hyperbolic sets for smooth maps: fine statistical properties. \emph{Journal of Differential Geometry}, 79(3), 433-477.  
\bibitem{Hartree} Hartree, D. R. (1928, January). The wave mechanics of an atom with a non-Coulomb central field. Part I. Theory and methods. In Mathematical \emph{Proceedings of the Cambridge Philosophical Society} (Vol. 24, No. 1, pp. 89-110). Cambridge university press.
\bibitem{Kan93} Kaneko K. (Ed.), Theory and Applications of Coupled Map Lattices, Wiley, 1993.
\bibitem{K00} Keller, G. An ergodic theoretic approach to mean field coupled maps. Fractal geometry and stochastics II. Birkh\"auser, Basel, 2000. 183--208.
\bibitem{KL04} Keller, G., \& Liverani, C. (2006). Uniqueness of the SRB measure for piecewise expanding weakly coupled map lattices in any dimension. \emph{Communications in Mathematical Physics}, 262(1), 33--50. 
\bibitem{KL09} Keller, G., \& Liverani, C. (2009). Map lattices coupled by collisions. \emph{Communications in Mathematical Physics}, 291(2), 591-597.
\bibitem{Kur84} Kuramoto, Y. (1984). Chemical turbulence. In \emph{Chemical oscillations, waves, and turbulence} (pp. 111-140). Springer, Berlin, Heidelberg.
\bibitem{MV}  Mouhot, C., \& Villani, C. (2011). On {L}andau damping. \emph{Acta Mathematica}, 207(1), 29-201.
\bibitem{PS91} Pesin Ya, B., \& Sinai Ya, G. (1991). Space-time chaos in chains of weakly interacting hyperbolic mappings. \emph{Adv. Sov. Math}, 3, 165-98.
\bibitem{P20} Pereira, T., van Strien, S., \& Tanzi, M. (2020). Heterogeneously coupled maps: hub dynamics and emergence across connectivity layers. Journal of the European Mathematical Society, 22(7), 2183-2252.
\bibitem{Spo} Spohn, H. (2012). \emph{Large scale dynamics of interacting particles}. Springer Science \& Business Media.
\bibitem{ST21} S\'elley, F. M., \& Tanzi, M. (2021). Linear response for a family of self-consistent transfer operators. \emph{Communications in Mathematical Physics}, 382(3), 1601-1624.
\bibitem{V68} Vlasov, A. A. (1968). The vibrational properties of an electron gas. \emph{Soviet Physics Uspekhi}, 10(6), 721.
\end{thebibliography}
\end{document}